\documentclass[a4paper]{amsart}
\usepackage{hyperref}

\usepackage[usenames,dvipsnames,svgnames,table]{xcolor}
\usepackage[all]{xy}
\usepackage{enumitem}
\usepackage[normalem]{ulem}
\newcommand{\tsk}[1]{\textcolor{YellowOrange}}
\usepackage{tikz}
\usepackage{tikz-cd}
\usepackage{booktabs}

\makeatletter
\def\@endtheorem{\endtrivlist}
\makeatother

\usepackage[active]{srcltx}
\usepackage{amsmath}
\usepackage{amssymb}
\usepackage{amscd}
\usepackage{amsthm}
\usepackage[latin1]{inputenc}
\usepackage{mathrsfs}  

\usepackage{nicefrac}

\newtheorem{teo}{Theorem}[section]

\newtheorem{defin}[teo]{Definition}
\newtheorem{prop}[teo]{Proposition}
\newtheorem{cor}[teo]{Corollary}
\newtheorem{lemma}[teo]{Lemma}

\theoremstyle{definition}

\newtheorem{remark}[teo]{Remark}

\newtheoremstyle{dico}
{\baselineskip}   
{\topsep}   
{}  
{0pt}       
{} 
{.}         
{5pt plus 1pt minus 1pt} 
{}          
\theoremstyle{dico}
\newtheorem{say}[teo]{}
\numberwithin{equation}{section}

\newcommand{\ra}{\rightarrow}
\newcommand{\C}{\mathbb{C}}

\newcommand{\Zeta}{{\mathbb{Z}}}

\newcommand{\meno}{^{-1}}

\newcommand{\alfa}{\alpha}
\newcommand{\alf}{\alpha}
\newcommand{\vacuo}{\emptyset}

\newcommand{\la}{\lambda}
\newcommand{\restr}[1]          {\vert_{#1}}
\newcommand{\Aut}{\operatorname{Aut}}
\newcommand{\Out}{\operatorname{Out}}

\renewcommand{\setminus}{-}

\newcommand{\eps}{\varepsilon}
\renewcommand{\phi}{\varphi}
\newcommand{\lds}{\ldots}
\newcommand{\cds}{\cdots}
\newcommand{\cd}{\cdot}
\newcommand{\im}{\operatorname{Im}}
\newcommand{\sx}{\langle}
\newcommand{\xs}{\rangle}

\newcommand{\lra}{\longrightarrow}
\newcommand{\tra}{\twoheadrightarrow}
 
\newcommand{\ga}{\gamma}
\newcommand{\Ga}{\Gamma}

\newcommand{\id}{\operatorname{id}}

\newcommand{\PP}{\mathbb{P}}   
   
\renewcommand{\phi}             {\varphi}

\newcommand{\M}{\mathsf{M}}

\newcommand{\A}{\mathsf{A}}

\newcommand{\PGL}                {\operatorname {PGL}}

\renewcommand{\Im}              {\operatorname{Im}}

\newcommand{\mihi}[1]{}

\newcommand{\barx}{{ \bar{x}}}

\newcommand{\Tri}{\operatorname {T}}

\newcommand{\ove}[1]{\overline{#1}}

\def\LaTeX{%
	\let\Begin\begin
	\let\salta\relax
	\let\finqui\relax
	\let\futuro\relax}
\LaTeX

\def\Edef{\end{definition}}

\def\Elem{\end{lemma}}

\def\Edim{\end{proof}}

\def\Eprop{\end{prop}}

\def\Ecor{\end{corollary}}

\def\Ethm{\end{theorem}}

\def\Erem{\end{remark}}

\newcommand{\puno}{\mathbb{P}^1}
\newcommand{\confpn}{\operatorname{\mathbf{F}}_{0,n}\mathbb{P}^1}
\newcommand{\inn}{\operatorname{inn}}
\newcommand{\Inn}{\operatorname{Inn}}

\newcommand{\baralf}{\bar{\alf}}

\newcommand{\auta}{\Aut^* \Ga_n}
\newcommand{\aug}{\Aut G}
\newcommand{\outa}{\operatorname{Out}^*\Ga_n}

\newcommand{\data}{\mathscr{D}}
\newcommand{\datan}{\mathscr{D}^n}

\newcommand{\ts}{\tilde{\sigma}}

\newcommand{\Datum}{\theta}

\newcommand{\cc}{\mathscr{C}}

\newcommand{\base}{{\mathscr{B}}}
\newcommand{\bbase}{{\bar {\base}}}
\newcommand{\baseb}{{\overline{\mathscr{B}}}}
\newcommand{\basef}{{\mathsf{Y}}}

\newcommand{\Mod}{\operatorname{Mod}}
\newcommand{\PMod}{\operatorname{PMod}}
\newcommand{\relo}{\,  \operatorname{rel}\, \{0,1\}}
\newcommand{\omo}{\simeq}
\newcommand{\dire}{\varepsilon}  
\newcommand{\tdire}{\tilde{\dire}}
\newcommand{\morf}{f} 
\newcommand{\morp}{\phi}   
\newcommand{\morff}{f'} 
\newcommand{\talfa}{\tilde{\alfa}}
\newcommand{\cxx}{\C^{**}}
\newcommand{\ccxx}{\conf_{0,n-3}\cxx}
\newcommand{\mon}{\M_{0,n}}
\newcommand{\monu}{\M_{0,n+1}}
\newcommand{\conf}{\operatorname{\mathbf{F}}}
\newcommand{\confn}{\conf_{0,n}}
\newcommand{\topo}{\mathscr{T}}

\newcommand{\autss}{\Aut^{**}}

\newcommand{\coll}{{\mathcal{I}}}
\newcommand{\collf}{\coll(*,f)}
\newcommand{\kolf}{\mathfrak{K}}
\newcommand{\burp}{\psi}
\newcommand{\ax}{a}
\newcommand{\basa}{\basef_\ax}
\newcommand{\yy}{y}
\newcommand{\ccs}{\cc^*}
\newcommand{\uu}{u}
\newcommand{\tv}{\bar{v}}
\newcommand{\ttv}{\tilde{v}}

\newcommand{\bara}{{\bar{a}}}
\newcommand{\barb}{{\bar{b}}}
\newcommand{\z}{{z}}
\newcommand{\w}{{w}}

\begin{document}
\author{Alessandro Ghigi, Carolina Tamborini}
\title[Families of Galois covers of the line]{A topological construction of families of Galois covers of the line}



\address{Universit\`{a} di Pavia}
\email{alessandro.ghigi@unipv.it}
\address{Universiteit Utrecht}

\email{c.tamborini@uu.nl} 
%
%
%
%

\begin{abstract}
We describe a new construction of families of Galois coverings of the line using basic properties of configuration
spaces, covering theory, and the Grauert-Remmert Extension Theorem. Our construction provides an alternative to a previous construction due to Gonz\'alez-D\'{i}ez and Harvey (which uses Teichm\"uller theory and Fuchsian groups) and, in the case the Galois group is non-abelian, corrects an inaccuracy therein.  In the opposite case where the Galois group has trivial center, we recover some results due to  Fried and V\"olklein. 
\end{abstract}

\maketitle
\tableofcontents

\section{Introduction}

The object of this note are families of Galois coverings of the line.

Let $G$ be a finite group and let $C$ and $C'$ be smooth projective
curves over the complex numbers endowed with a $G$-action. We say that
$C$ and $C'$ are \emph{topologically equivalent} or have the same
(unmarked) \emph{topological type} if there are $\eta \in \Aut G$ and
an orientation preserving homeomorphism $f: C \ra C'$ such that
$f(g\cd x) = \eta (g) \cd f(x)$ for $x\in C'$ and $g\in G$. We say
that $C$ and $C'$ are (unmarkedly) $G$-\emph{isomorphic} if moreover
$f$ is a biholomorphism.

Given a $G$-covering $C\ra \puno$, it has been proved by
Gonz\'alez-D\'{i}ez and Harvey \cite{ganzdiez} that there exists an
algebraic family of curves with a $G$-action
\begin{gather*}
\pi: \cc \ra \mathsf{B},
\end{gather*}
such that \begin{enumerate}
\item every curve $C'$ in the family is \emph{topologically
equivalent} to $C$;
\item every curve with an action of the given topological type is
$G$-\emph{isomorphic} to some fiber of the family and at most to a
finite number of fibers.
\end{enumerate}
This result has been subsequently used in several papers,
e.g. \cite{cgp,fgp,fn,penego,perroni}, just to mention a few.

The construction in \cite{ganzdiez} uses Teichm\"uller theory.  Other
approaches to this construction include \cite{friedv,volk,li}.  In
this paper we describe an alternative, explicit and mostly topological
construction of such families. We expect this to be useful to make
explicit computations on the family. E.g., we expect this to allow a
better understanding of the monodromy and the generic Hodge group for
the natural variation of Hodge structure associated with the family,
generalizing the results of \cite{Rohde} carried out in the cyclic
case.  Our motivation comes from the fact that these families and
their variation of the Hodge structure are important in the study of
Shimura subvarieties of the moduli space $\A_g$ (of principally
polarized abelian varieties of dimension $g$) in relation with the
Coleman-Oort conjecture (see e.g. \cite{moonen-special,
moonen-oort,fgp,t}). The results presented here are, nevertheless,
of indipendent interest.

\begin{say}
We give a quick glance at our construction.  For $n\geq 3$ let
$ \M_{0,n}$ denote the set of $n$-tuples
$X=(x_1, \lds, x_n) \in (\PP^1)^n$ such that $x_i \neq x_j $ for
$ i\neq j$ and $ x_{n-2}= 0, x_{n-1} = 1,x_n = \infty$. Consider the
group
$\Gamma_n=\sx \gamma_1,...,\gamma_n|\ \gamma_1 \cds \ga_n=1\xs$.
Let $G$ be a finite group and let $\theta: \Gamma_n\ra G$ be an
epimorphism.  Fix $X\in \M_{0,n}$.  After choosing a base point
$x_0 \in \PP^1 -X$ and an isomorphism
$\chi:\Ga_n \cong \pi_1(\PP^1-X, x_0)$, Riemann's existence theorem
yields a $G$-covering $C_X \ra \puno$ with monodromy
$\theta\circ \chi\meno$ and branch locus $X$. Nevertheless this
covering depends on the choices. Our goal is to make this
construction for all $X\in \mon$ together, in order to get a family
of curves parametrized by $\mon$. Consider the map
\begin{gather*}
p: \M_{0,n+1} \ra \M_{0,n}, \quad p(x_0,
x_1,...,x_n)=(x_1,...,x_n).
\end{gather*}
We have $p^{-1}(X)=\mathbb{P}^1\setminus X$.  Hence $p$ can be
thought as the universal family of genus $0$ curves with $n$ marked
points.  The basic idea of our construction is that the total space
of our family should be a suitable $G$-covering of $\M_{0,n+1}$.
For the construction of this covering, choose
\begin{enumerate}
\item[i)] an element $x=(x_0, X)\in \M_{0,n+1}$;
\item[ii)] an isomorphism $\chi: \Gamma_n \ra \pi_1(\puno-X,
x_0)$. 
\end{enumerate}
The following sequence is exact and splits:
\begin{gather*}
1\ra \pi_1(\mathbb{P}^1\setminus X, x_0) \ra \pi_1(\M_{0,n+1},
x)\ra \pi_1(\M_{0,n}, X)\ra 1.
\end{gather*}
Set for simplicity $N_x: = \pi_1(\mathbb{P}^1\setminus X, x_0)$,
$H_X:= \pi_1(\M_{0,n}, X)$
and $f:=\chi^{-1}\circ \theta$.  Assume that we can find an
extension $\tilde{f}$:
\begin{equation*}
\begin{tikzcd}
1 \arrow[r, ]{}& \pi_1(\mathbb{P}^1\setminus X, x_0) \arrow[dr,
"f"] \arrow[r, ]{} & \pi_1(\M_{0,n+1}, x)\arrow[r, ]{} \arrow[d,
"\tilde{f}",dashed ]{} & H_X
\arrow[r, ]{} & 1.\\
& & G & 
\end{tikzcd}
\end{equation*}
From $\tilde{f}$ we get a topological $G$-covering
$ \cc^* \ra \M_{0,n+1} $. By Grauert-Remmert Extension Theorem (see
Theorem \ref{grr} below) this compactifies to a branched covering
$ \cc \ra \puno \times \M_{0,n}$ of quasi-projective varieties.
Composing with the projection to $\M_{0,n}$ we get a holomorphic
family $ \pi:\cc\ra \M_{0,n}$ satisfying property (1) and (2).
\end{say}

\begin{say}
Thus, if one is able to find the extension $\tilde{f}$, one can
construct the families using only basic properties of configuration
spaces, covering theory and the Grauert-Remmert Extension Theorem,
avoiding Teichm\"uller theory and Fuchsian groups.  In fact this
strategy is not new, as it has already been used in exactly the same
context in various papers by Michael D. Fried and Helmut V\"olklein,
see e.g.  \cite {fried77,friedv,volk}.

If $G$ is abelian, one is always able to find the extension
$\tilde{f}$, see Section \ref{sec:abelian}.  In general however the
extension $ \tilde{f}$ does not exist, contrary to what claimed in
\cite{ganzdiez}. One can at least show that there are always finite
index subgroups $H_a\subset H_X$ such that $f$ extends to a morphism
$f_a:N_x\rtimes H_a\ra G $.  Geometrically passing from $H_X$ to the
subgroup $H_a$ means that one builds a family satisfying (1) and (2)
over a base which is not any more $\M_{0,n}$, but some finite cover
$\basef_{\ax}$ of it.  The pair $(H_a,f_a)$ is far from unique,
there are many of them and different choices yield families
differing by finite \'etale pull-back (see Section
\ref{sec:famG-curve} for precise definitions.)  So another problem
arises: how is one supposed to choose the pair $(H_a,f_a)$ in order
to determine the family in a canonical way?

For a special class of groups, namely for groups $G$ with trivial
center, there is a canonical choice of $(H_a,f_a)$, which allows to
construct a canonical family of coverings. This case corresponds to
the one studied in \cite{friedv,volk,volkl} where the condition that
$G$ be centerless plays a crucial role.

It is odd that for this problem the two special cases occur in
opposite directions, namely for abelian and for centerless groups.

For general $G$ one is not able to pick out a
distinguished choice in a canonical way. This problem was already considered long ago in  \cite[p. 57-58]{fried77}
where a  cohomological interpretation  of this difficulty is given.

Our approach instead is the following. Since we are stuck with a whole
collection of pairs $(H_a,f_a)$, each one giving rise to a family of
coverings with base the cover $\basef_\ax$ of $\mon$, we decide to
consider the whole collection instead of the single families.  This
collection comes naturally with the structure of a directed set coming
from the pull-backs among families.  We are able to show that this
collection with this structure is well-defined and depends only on the
topological data.

Summing up, our construction, which builds heavily on previous
approaches, corrects an inaccuracy in \cite{ganzdiez}, where it is
erroneously claimed that one has always $\basef_a=\mon$, confirms
that $\basef_a=\mon$ if $G$ is abelian (Theorem \ref{casoabeliano}),
and allows to recover at least part of the results in the papers of
Fried and V\"olklein quoted above, while generalizing them to
arbitrary groups with nontrivial center.	
\end{say}

\begin{say}
The paper is organized as follows. In Section
\ref{sec:configuration-spaces} we recall basic facts about the
configuration spaces of $\PP^1$. Section \ref{sec:par-transp} deals
with parallel transport for fiber bundles. This material is for sure
known to the experts, but rather hard to locate in the
literature. Since these arguments are quite useful and we like their
geometric flavour, we prefer to expound them concisely. In Section
\ref{sec:geometric-bases} we recall some classical concepts of
surface topology. After these preliminaries in Section
\ref{sec:tipitopo} we study the set $\topo^n(G)$ of topological
types of $G$-actions: the main result is Theorem \ref{teotipitopo},
that gives a combinatorial description of the set of topological
types.  The proof of this well-known fact presents our ideas in a
simple context.  Section \ref{sec:groups} is dedicated to the
description of some technical tools for the construction of the
families. In Section \ref{sec:famG-curve} we contruct the collection
of families $\{\cc_{\ax} \ra \basef_{\ax}\}_a$ as sketched above.
In \ref {sec:indip-choices} we study the dependence of the
collection on the choices (i) and (ii), and on the epimorphism
$\theta: \Gamma_n\ra G$. Also this point becomes quite neat using
our approach. Section \ref{sec:centerless} is dedicated to the case where
$G$ has trivial center and Section \ref{sec:abelian} to case where
$G$ is abelian. Summing up our main Theorem is the following:
\end{say}
\begin{teo}
\label{mainz}\phantom{strudel}
\begin{enumerate}
\item The topological types of $G$-curves $C$ with $g(C)=g$,
$g(C/G)= 0$ and $n$ branch points are in bijection with the set
$\data^n(G) / \Aut G \times \outa$ (see Corollary \ref{outout} and
Definitions \ref{tipitopo}, \ref{sayautstar} and \eqref{eq:5} for
notation).
\item For any topological type there is a non-empty ordered set
$(\coll,\geq)$ and for any $a\in \coll$ there is an algebraic
family $\pi_a: \cc_a \ra \basef_a$ of genus $g$ curves with a
$G$-action. The following properties hold:
\begin{enumerate}
\item every curve $C$ in the family has the given topological
type;
\item for any $a\in \coll$ and for any $G$-curve $C$ with
$C/G\cong \PP^1$, there is at least one fibre of
$\pi_a: \cc_a \ra \basef_a$ which is $G$-isomorphic to $C$, and
there are only finitely many such fibres;
\item each $\basef_a$ is a finite \'etale cover of $\mon$;
\item $(\coll,\geq)$ is a directed set: for any $a,b \in \coll$
there is $c$ with $c\geq a$ and $ c \geq b$;
\item if $a\geq b$, there is an algebraic \'etale covering
$v:\basef_a \ra \basef_b $ such that $ \cc_a \cong v^* \cc _b$;
\item all the families have the same moduli image;
\item  if $Z(G)=\{1\}$, then $(\coll,\geq)$ has a minimum;
hence in this case we can associate to any topological type a
single family instead of the whole collection;
\item if $G$ is abelian, then there exists $a\in \coll$ such that
$Y_{a} = \mon$.
\end{enumerate}
\end{enumerate}
\end{teo}
The precise statement can be found in Theorems \ref{costruzione} and
\ref{maindep}.  Roughly speaking one can say that for any topological
type there is a ``universal'' family of $G$-curves with that
topological type. Such a family is not unique, but only unique up to
the equivalence relation generate by finite \'etale pull-backs.

\begin{say}
The existence problem that we address in this paper can of course
be generalized: instead of considering just Galois covers, one can
ask for the construction of families satisfying (1) and (2) for
all the coverings with a fixed Galois closure (equivalently with
fixed monodromy). This kind of problems have been studied a lot
and they are extremely important also because of their relevance
for the inverse Galois problem, see
\cite{fried77,friedv,friedann,friedjarden,fried10,volkl}.  In
these cases it often happens that the ``universal'' family has
more than one component. We stress that in this paper we restrict
only to the Galois case and that in this case all families are
connected. Infact the base of each family is a (connected) cover
of $\mon$.

Another variant of the problem studied in this paper is obtained
by letting $G^*$ be a group such that
$\Inn G \subset G^* \subset \Aut G$ and considering equivalent two
data if and only if they belong to the same
$G^*\times \outa$-orbit. This also has attracted a lot of
attention in the literature. Our case corresponds to the choice
$G^*=\Aut G$. In this paper we restrict to this case since we are
interested in the topological types.

\end{say}

\medskip

{\bfseries \noindent{Acknowledgements}}.  The authors would like to
thank Michael D. Fried, Gabino Gonz\'alez-D\'iez and Fabio Perroni for
useful discussions/emails related to the subject of this paper.  We
are very grateful to the anonymous referee for several interesting
questions, in particular for pushing us to study the case of a
centerless group, see Section \ref{sec:centerless}. The authors were partially supported by MIUR PRIN 2017
``Moduli spaces and Lie Theory'', by MIUR, Programma Dipartimenti
di Eccellenza (2018-2022) - Dipartimento di Matematica
``F. Casorati'', Universit\`a degli Studi di Pavia and by INdAM
(GNSAGA). The second author was partially supported by the Dutch Research Council (NWO grant BM.000230.1).

\section{Configuration spaces}
\label{sec:configuration-spaces}

\begin{say}
If $M$ is a manifold, its configuration space is
\begin{gather*}
\confn M := \{ (x_1, \lds, x_n ) \in M^n: x_i \neq x_j \text{ for
} i\neq j\}.
\end{gather*}
We use the following notation: $X = (x_1, \lds, x_n)$ is a point of
$ \conf_{0,n}M$ and $x =(x_0, X)= (x_0,x_1, \lds, x_n) $ is a point
of $ \conf_{0,n+1}M $.  We set $M-X: =M- \{x_1\cds, x_n\}$.  The
group $\pi_1 (\confn M)$ is called the \emph{pure braid group} with
$n$ strings of the manifold $M$.

\end{say}

\begin{say}
\label{omegan}
If $n\geq 3$, the group $\PGL(2,\C)$ acts freely and holomorphically
on $\conf_{0,n}\PP^1$.  The quotient $ \confpn / \PGL(2,\C) $ is the
moduli space of smooth curves of genus $0$ with $n$ marked
points. Set $ \cxx:=\C\setminus\{0,1\}$.  The map
\begin{gather*}
\conf _{0,n-3} \cxx \lra \confpn, \quad (z_1, \lds, z_{n-3} )
\mapsto (z_1, \lds, z_{n-3},0, 1, \infty)
\end{gather*}
is a section for the action of $\PGL(2,\C)$, i.e. its image
intersects each orbit in exactly one point and it induces a
biholomorphism of $\ccxx $ onto the moduli space
$ \confpn / \PGL(2,\C)$.  We \emph{define} $\mon$ as the image of
the section i.e. we set
\begin{gather*}
\M_{0,n}: =  \ccxx \times   \{(0,1,\infty) \}   = \\
=\{X=(x_1, \lds, x_n) \in \confpn : x_{n-2}= 0, x_{n-1} = 1,x_n =
\infty\}.
\end{gather*}
Points of $\mon$ will be denoted as
$X= (x_1, \lds, x_n)$ with the understanding that
$x_{n-2}= 0 , x_{n-1} = 1, x_n = \infty$.  Similarly we set
\begin{gather*}
\M_{0,n+1}: = \{x=(x_0, \lds, x_{n+1}) \in \conf _{0,n+1} \PP^1:
x_{n-2}= 0, x_{n-1} = 1,x_n = \infty\}.
\end{gather*}
It is often useful to compare the configuration space of $\PP^1$
with that of the plane. Denote by $\Tri(2, \C)\subset \PGL(2, \C)$
the subset of elements in $\PGL(2, \C)$ fixing $\infty$. The group
$\Tri(2, \C)$ acts on $\conf_{0,n-1}\C$ and the map
\begin{gather}\label{biolom}
\mon \lra \conf_{0,n-1}\C, \quad X \mapsto (x_1, \lds, x_{n-3},0,
1)
\end{gather}
is a section for this action, hence
$\M_{0,n}\times \Tri(2, \C)\cong\conf_{0,n-1}\C$. In particular,
$\pi_1(\M_{0,n})\subset \pi_1(\conf_{0,n-1}\C)$. Thus, when dealing
with $\pi_1(\M_{0,n})$, we can work with the more classical braid
group of the plane.  The map
\begin{gather}
\label{eq:1}
p: \monu \ra \mon, \quad p(x_0, X ) := X ,
\end{gather}
is a fiber bundle.
In fact it is the restriction of the bundle $\conf_{0,n}\C \ra \conf_{0,n-1}\C$, see \cite{birman}.
The fibre  over $X $ is
$\PP^1 - X = \cxx - \{x_1, \cds, x_{n-3} \} $.
Hence \eqref{eq:1} is
the universal family of genus $0$ curves with $n$ ordered marked
points.
\end{say}

\begin{say}\label{succesioniesatte}
Fix $x=(x_0, X)\in \M_{0,n+1}$ and let
$\tilde x=(x_0, \tilde X)\in \conf_{0,n}\C$ be the corresponding
point via \eqref{biolom}.
We have a commutative diagram
\begin{equation}
\label{seqpuni}
\begin{tikzcd}[row sep=small, column sep=small]
1 \arrow[r, ""] &\pi_1(\puno-X, x_0) \arrow[r, ""]\arrow[d, equal]&\pi_1( \M_{0,n+1}, x) \arrow[r, ""] \arrow[d, hook]&\pi_1(	\M_{0,n}, X) \arrow[r, ""]\arrow[d, hook]& 1 \\
1\arrow[r, ""] & \pi_1(\C \setminus \tilde X, x_0) \arrow[r,
""] & \pi_1(\conf_{0,n}\C,\tilde x)\arrow[r, ""] &
\pi_1(\conf_{n-1}\C, \tilde X)\arrow[r, ""] & 1.
\end{tikzcd}
\end{equation}

The rows are the split exact sequence of the fibrations $p$ and
$\conf_{0,n}\C \ra \conf_{0,n-1}\C$, see e.g. \cite[Corollary
1.8.1]{birman} and \cite[Theorem 3.1]{fadell}.  A geometric way to
exhibit the splitting is to produce a cross section as follows: given
$x = (x_1, \lds, x_n) \in \mon$ we set
\begin{gather*}
f(x) : = \frac{1}{2} \min \{ 1, | x_1 |, \lds , |x_{ n-3}|\}.
\end{gather*}
Then $s (x) := ( f(x), x_1, \lds, x_n )$ is a section of
$p: \monu \ra \mon$.  (A similar idea is used in
\cite[Thm. 3.1]{fadell}.)  The morphism
$s_*: \pi_1(\M_{0,n}, X)\ra \pi_1(\M_{0,n+1}, x)$ is a splitting.
Setting
\begin{equation}\label{epsilon}
\begin{split}
&\varepsilon: \pi_1(\M_{0,n}, X)\ra \Aut(\pi_1(\puno-X, x_0)) \\
&\varepsilon([\alpha])([\gamma]):=s_*[\alpha] \cdot [\gamma]\cdot
s_*[\alpha]^{-1}= [s \circ \alpha \ast \gamma \ast s\circ
i(\alpha)].
\end{split}
\end{equation}
we get
\begin{gather*}
\pi_1(\M_{0,n+1}, x) = \pi_1(\puno-X, x_0)) \rtimes_\eps \pi_1(\mon, X).
\end{gather*}
\end{say}

\section{Parallel transport}\label{sec:par-transp}
In this section we recall a notion of parallel transport up to
homotopy on any fiber bundle. In the sequel, we will use it for the
bundle $p:\M_{0,n+1}\ra \M_{0,n}$ to study the dependence of the
construction of Section \ref {sec:famG-curve} from the choices made.

\begin{say}
Given $b_0, b_1 \in B$ let $\Omega(B,b_0,b_1)$ denote the set of all
paths $\alfa$ in $B$ with $\alfa(0)=b_0$ and $\alfa(1)=b_1$.  We
write $\alfa \sim \beta$ if $\alfa \omo \beta \relo$.  Let
$\Pi_1(B)$ denote the fundamental groupoid of $B$: this is the small
category whose objects are the points of $B$ and with morphisms from
$b_0$ to $b_1$ equal to $\Omega(B,b_0,b_1) / \sim$, composition
being given by $[\alfa]\cd [\beta] = [\alfa * \beta]$.
\end{say}

\begin{say} \label{Htilde} Let $p: E \ra B$ be a fiber bundle (in the
sense of \cite[p.90]{spanier} i.e. a locally trivial bundle). Assume
that the base $B$ is Hausdorff and paracompact. Then $p$ is a
fibration \cite[Cor. 14, p. 96]{spanier} i.e. it has the homotopy
lifting property for every topological space Z: if
$H: Z \times [0,1] \ra B$ is any map and $f : Z \ra E$ lifts
$H(\cd, 0)$, then there is a lifting $\tilde{H}$ of $H$ with
$\tilde{H}(\cd, 0) = f$ (see e.g.  \cite[p. 66]{spanier}). For any
fiber bundle $p: E \ra B$ one can define a sort of parallel
transport up to homotopy, which is a contravariant functor $T$ from
$\Pi_1(B)$ to the homotopy category of topological spaces, denoted
by $\mathrm{h-TOP}$. For $b\in B$ set $T(b):= E_b =
p\meno(b)$. Given $[\alfa] \in \Pi_1 (B) (b_0, b_1) $ consider the
map $H: E_{b_0} \times [0,1]\ra B, H(e,t):=\alfa(t)$. The inclusion
$i: E_{b_0} \hookrightarrow E$ is a lifting of $ H(\cd, 0)$. By the
homotopy lifting property there is
$\tilde{H}: E_{b_0} \times [0,1] \ra E$ with $p\tilde{H} = H$ and
$\tilde{H}(\cd, 0) = i$.  Moreover the homotopy class of
$\tilde{H}(\cd , 1)$ is well-defined. 
We call
$T([\alfa]) = [ \tilde {H} ( \cd, 1)]\in [ E_{b_0}, E_{b_1}]$ the
homotopy parallel transport along $\alfa$, see e.g.
\cite[p. 100f]{spanier} and \cite[p. 54]{may}.
\end{say}
\begin{say}
If $p: E \ra B$ is a differentiable fiber bundle one can say more.
Recall the following basic fact from differential topology.  Let $M$
and $N$ be smooth manifolds. An \emph{isotopy} of $M$ in $N$ is a
smooth map $f: M \times [0,1] \ra N$ such that $f(\cd, t)$ is an
embedding for any $t$.  If $M=N$, $f(\cd, t)$ is a diffeomorphism of
$M$ for any $t$ and $f(\cd, 0) = \id_M$, we say that $f$ is a
\emph{ambient isotopy}.
\end{say}
\begin{teo} \label{Hirsch} If $M$ is a compact submanifold of $N$, any
isotopy $f: M \times [0,1] \ra N$ such that $f(\cd , 0)$ is the
inclusion $M \hookrightarrow N$ extends to an ambient isotopy.
\end{teo}
(See e.g \cite[Thm. 1.3 p. 180]{hirsch}.)

\begin{lemma}
\label{lemmatrasporto}
Assume that $p : E \ra B$ is a differentiable bundle. Let $\alfa $
be a path in $B$ from $b_0$ to $ b_1$. Let $\sigma $ be a path in
$E$ with $p \sigma = \alfa$ and set $x_0 = \sigma(0) \in E_{b_0}$,
$x'_0 = \sigma(1)\in E_{b_1}$.  Then there is a map
$\tilde{H} : E_{b_0} \times [0,1] \ra E$ such that
\begin{enumerate}
\item $\tilde{H} (\cd ,0) $ is the inclusion
$E_{b_0} \hookrightarrow E$;
\item $\tilde{H}(\cd, t)$ is a diffeomorphism of $E_{b_0}$ onto
$E_{\alfa(t)}$;
\item $ \tilde{H}(x_0 , t) = \sigma(t)$.
\end{enumerate}
In particular the map $f^{\alpha}:= \tilde{H}(\cd, 1)$ is a
diffeomorphism of $E_{b_0}$ onto $E_{b_1}$ such that $f^\alf(x_0) = x_0'$
and $T([\alfa]) = [f^{\alpha}]$.  Moreover if $G$ is a finite group
acting fibrewise on $E$ and the fibre is compact, then $f^\alfa$ can
be chosen $G$-equivariant.
\end{lemma}
\begin{proof}
If the fibre of $E$ is compact, the argument is the usual proof of
Ehresmann theorem: pull-back $E$ to $[0,1]$, choose a lifting to $E$
of the vector field $d/dt$ and integrate it, see
e.g. \cite{voisin}. A $G$-invariant lift gives the last
statement. But we also need the case of non-compact fibres. This can
be treated as follows.  Denote by $\talfa: \alfa^*E \ra E$ the
bundle map covering $\alfa$.  Since $[0,1]$ is contractible, there
is a (smooth) trivialization
$\psi: E_{b_0} \times [0,1] \ra \alfa^*E$ such that
$\psi (x,0) = x$, see \cite[Cor. 11.6 p. 53]{steenrod}.  Given any
such $\psi$ the composition
$\talfa \circ \psi :E_{b_0} \times [0,1] \ra E$ is a possible choice
for the map $\tilde{H} $ in \ref{Htilde}. We now modify $\psi$ so
that it matches the conditions (a)-(c). First notice that if
$\{h_t\}_{t\in [0,1]}$ is any path in
$\operatorname{Diff} (E_{b_0})$ starting at the the identity, then
$\psi'_t:= \psi_t h_t$ is a new trivialization of $\alfa^*E$.  Next
observe that $t\mapsto \psi_t^{-1}(\sigma(t))$ is a path in
$E_{b_0}$ from $x_0$ to $\psi_1\meno (x'_0)$, i.e.  an isotopy of
$\{x_0\}$ in $E_{b_0}$. By Theorem \ref{Hirsch} there is $\{h_t\}$
that extends this isotopy. Then $\psi'_t:= \psi_t h_t$ is a
trivialization and $\tilde{H}:= \talfa \circ \psi'$ satisfies
(a)-(c).
\end{proof}


We now use this construction for the fiber bundle
$\M_{0,n+1} \ra \M_{0,n}$ and give a geometric intepretation of the
morphism \eqref{epsilon} in terms of parallel transport.

\begin{prop}\label{sigmacancellet}
Let $x,x'\in \M_{0,n+1}$. Let $\beta:[0,1]\ra \M_{0,n}$ be a path
such that $\beta(0)=X$ and $\beta(1)=X'$. Let $\tilde H$,
$f^{\beta}$ and $T([\beta])$ be as in Lemma \ref{lemmatrasporto}.
Assume that $f^{\beta}(x_0)=x'_0$. Set
$\tilde \beta(t):=\tilde H(t,x_0)$. Then for
$[\gamma]\in \pi_1(\puno-X, x_0)$ we have
$f^{\beta}_*([\gamma])=\tilde{\beta}_{\#}([\gamma])$.
\end{prop}
\begin{proof}
Take $[\gamma]\in \pi_1(\puno-X, x_0)$. Consider the map
\begin{gather*}
F: [0,1]\times [0,1] \ra \monu\quad F(t,s)=\tilde H(\gamma(s), t).
\end{gather*}
Then $F(0,s)=\tilde H(\gamma(s),0)=\gamma(s)$,
$F(0,1)=\tilde H(\gamma(s), 1)=f^{\beta}\circ \gamma(s)$ and
$F(t,0)=F(t,1)=\tilde H(x_0,t)=\tilde\beta(t)$.  It follows that
$ i(\tilde{\beta}) \ast \gamma \ast \tilde{\beta} \simeq f^\beta
\circ \gamma\ \operatorname{rel} \{0,1\}$.  Hence
$f^{\beta}_*([\gamma])=\tilde{\beta}_{\#}([\gamma])$ for any
$[\gamma]\in \pi_1(\puno-X, x_0)$.
\end{proof}

\begin{prop}\label{esp-parall} Let $[\alpha]\in\pi_1(\M_{0,n}, X)$
and let $\tilde H$, $f^{\alpha}$ and $T([\alpha])$ be as in Lemma
\ref{lemmatrasporto}.  Assume that
$\sigma(t):=\tilde H(t,x_0)=s\circ \alpha$. Then
$\varepsilon([\alpha])=f^{\alpha}_*$.
\end{prop}
\begin{proof}
By Proposition \ref{sigmacancellet}, we get
$f^{\alpha}_*[\gamma]=\sigma_{\#}([\gamma])=[\sigma \ast \gamma \ast
i(\sigma)]$ for any $[\gamma]\in \pi_1(\puno-X, x_0)$. 
Hence $f^{\alpha}$ satisfies
$f^{\alpha}_*[\gamma]=[s \circ \alpha \ast \gamma \ast s\circ
i(\alpha)]=\varepsilon([\alpha])([\gamma])$ for every
$[\gamma]\in \pi_1(\puno-X, x_0)$. This proves the
statement.
\end{proof}

\section{Dehn-Nielsen theorems and consequences
}\label{sec:geometric-bases}

We dedicate this section to fixing some notations and recalling some
classical concepts of surface topology.

\begin{say} \label{sayhatalfa} Let $\Sigma$ be an oriented surface and
set $\Sigma^*:=\Sigma - \{y\}$ for some $y\in \Sigma$.  Given
$b_0, b_1 \in \Sigma$ let $\Omega(\Sigma,b_0,b_1)$ denote the set of
all paths $\alfa$ in $\Sigma$ with $\alfa(0)=b_0$ and
$\alfa(1)=b_1$. Fix $x_0\in \Sigma^*$.  Let
$\tilde\alfa \in \Omega (\Sigma, x_0,y)$ be such that
$\tilde\alfa(t) = y$ only for $t=1$ and let $D$ be a small disk
around $y$.  Let ${\alfa} $ be the loop that starts at $x_0$,
travels along $\tilde\alfa$ till it reaches $\partial D$, then makes
a complete tour of $\partial D$ counterclockwise and finally goes
back to $x_0$ again along $\tilde\alfa$.  An important observation
is that the conjugacy class of $[{\alfa}]$ in $\pi_1 (\Sigma^*,x_0)$
is well defined. Indeed the choice of the disk does not change
$[{\alfa}]$, while if a different path
$\tilde\beta\in \Omega(\Sigma, x_0,y)$ is chosen, then $[{\beta}]$
and $[{\alfa}]$ are conjugate by the class of a loop in $\Sigma^*$
that starts at $x_0$ travels along $\tilde\alfa$ up to $\partial D$,
then along a piece of $\partial D$ and finally goes back along
$\tilde\beta$.
\end{say}

\begin{say}
Fix a point $(x_0, X) \in \conf_{0,n}S^2$.  Consider a smooth
regular arc $\tilde{\alpha}_i$ joining $x_0$ to $x_{\sigma_i}$ (for
some permutation $\sigma$). Assume that the paths $\tilde{\alfa}_i$
intersect only at $x_0$ and that the tangent vectors at $x_0$ are
all distinct and follow each other in counterclockwise order (we
orient $S^2$ by the outer normal).  Now consider the loops $\alfa_i$
constructed as in \ref{sayhatalfa} and assume that the circles are
pairwise disjoint and that the intersection of the interior of the
$i$-th circle with $X$ reduces to $x_{\sigma_i}$.
\end{say}
\begin{defin}
\label{geometric-basis}
Let $x=(x_0, X)\in \conf_{0,n+1}S^2$. We call a set of generators
$\base=\{[\alfa_1], \lds, [\alfa_n]\}$ obtained as above a
\emph{{geometric basis}} of $\pi_1(S^2 -X,x_0)$. We say that a
geometric basis $\base = \{[\alfa_i]\}_{i=1}^n$ is \emph{{adapted}}
to $x$ if it respects the order of the points in $X$, that is,
$\alpha_i$ turns around $x_i$, i.e., the permutation $\sigma=id$.
\end{defin}

Notice that, thanks to the permutation, the definition of geometric
basis depends only on the set $\{x_1, \lds , x_n\}$, not on the
ordering of the points. On the other hand the classes $\{[\alfa_i]\}$
have a fixed order.

\begin{say}\label{bbbendef}
For $n \geq 3$, set
$\Ga_{n}:= \langle \gamma _1, \dots, \gamma _n \ \vert \ \prod
_{i=1} ^n \gamma _i = 1 \rangle.$ From a geometric basis
$\mathscr{B} = \{[\alfa_i]\}_{i=1}^n$ we get an isomorphism
\begin{gather*}
\chi: \Ga_n \ra \pi_1(S^2 - X, x_0)
\end{gather*} such that $\chi(\ga_i) = [\alfa_i]$.  Assume that
$\base = \{[\alfa_i]\}_{i=1}^n$ and $\baseb=\{[\baralf_i]\}_{i=1}^n$
are two geometric bases for $\pi_1(S^2 -X, x_0)$. It follows from
\ref{sayhatalfa} that every $[\alfa_i]$ is conjugate to some
$[\baralf_j]$.  If we denote by
$\chi,\overline{\chi}: \Ga_n \ra \pi_1(S^2 -X, x_0)$ the
isomorphisms corresponding to the two bases, then
$\mu:=\overline{\chi} \circ \chi\meno \in \Aut \pi_1(S^2 -X,x_0)$
has the following properties:
\begin{enumerate}
\item for every $i = 1, \lds, n$, $\mu([\alfa_i])$ is conjugate to
$[\alfa_j]$ for some $j$;
\item the induced homomorphism on $H^2(\pi_1(S^2-X,x_0), \Zeta)$ is
the identity.
\end{enumerate}
\end{say}

\begin{defin}
We denote by $\Aut^* \pi_1(S^2 -X,x_0)$ the subgroup of elements of
$\Aut \pi_1(S^2 -X,x_0)$ satisfying properties (1) and (2) above.
By \ref{bbbendef} this definition does not depend on the choice of
the geometric basis $\base$.
\end{defin}

\begin{say}
Now assume that $\base$ and $\baseb$ are adapted to $X$. In this
case, for every $i = 1, \lds, n$, $[\alfa_i]$ is conjugate to
$[\baralf_i]$. As a consequence, the automorphism
$\mu:=\overline{\chi} \circ \chi\meno$ of $\pi_1(S^2 -X,x_0)$
belongs to the subgroup $\Aut^{**}\pi_1(S^2 -X,x_0)$ defined as
follows

\end{say}

\begin{defin}
We denote by $\Aut^{**} \pi_1(S^2 -X,x_0)$ the subgroup of
$\Aut^* \pi_1(S^2 -X,x_0)$ of elements that map $[\alfa_i]$ to a
conjugate of $[\alfa_i]$ for every $i = 1, \lds, n$.  The definition
does not depend on the choice of the geometric basis $\base$ adapted
to $x$.
\end{defin}

\begin{defin} \label{sayautstar} Similarly, we denote by
$\auta \subset \Aut \Ga_n$ the subgroup of automorphisms $\nu$
satisfying:
\begin{enumerate}
\item for $i=1, \lds, n$ the element $\nu (\ga_i)$ is conjugate to
$\ga_j$ for some $j$;
\item the automorphism of $H^2(\Ga_n,\Zeta)$ induced by $\nu$ is the
identity.
\end{enumerate}
We denote by $\Aut^{**}\Gamma_n\subset \Aut^* \Ga_n$ the subgroup of
automorphisms $\nu$ such that

\begin{enumerate}
\item [(1')] for $i=1, \lds, n$ the element $\nu (\ga_i)$ is
conjugate to $\ga_i$.
\end{enumerate}
\end{defin}

If $\chi: \Ga_n \ra \pi_1(S^2 -X, x_0)$ is induced from
a geometric basis (not necessarily adapted to $x$), then
$\nu \in \auta $ (resp., $\Aut^{**}\Gamma_n $) if and only if
$\chi \nu \chi\meno \in \Aut^*\pi_1(S^2-X,x_0)$ (resp.
$\Aut^{**}\pi_1(S^2-X,x_0)$).

\begin{say}
\label{inner}
If $G$ is a group and $a\in G$, then $\inn_a: G \ra G$ denotes
conjugation by $a$: $\inn_a(x) = axa\meno$. Notice that if
$f : G \ra H$ is a morphism, then
$f \circ \inn_a = \inn_{f(a)}\circ f$.  The group of inner
automorphisms of $G$ is denoted $\Inn G$. It is a normal subgroup of
$\Aut G$. We set $\Out G:= \Aut G / \Inn G$. For
$(x_0,X)\in \conf_{0,n+1}S^2$, we
observe that $\Inn(\pi_1(S^2-X, x_0)) \subset
\Out^{**}(\pi_1(S^2-X, x_0))$ and
$\Inn \Ga_n \subset \Aut^{**}\Ga_n$ and 
we define
\begin{gather}
\begin{gathered}
\Out^*\pi_1(S^2-X, x_0):=\frac{\Aut^*\pi_1(S^2-X,
x_0)}{\Inn\, \pi_1(S^2-X, x_0)}  \\
\Out^{**}\pi_1(S^2-X, x_0):=\frac{\Aut^{**}\pi_1(S^2-X,
x_0)}{\Inn\, \pi_1(S^2-X, x_0)}
\end{gathered}
\qquad \label{eq:5}
\outa: = \frac{\auta}{\operatorname{Inn} \Ga_r}
\end{gather}
Using a geometric basis we immediately get
$\outa \cong \Out^*\pi_1(S^2-X, x_0)$.
\end{say}

\begin{say}If $S_{g,n}$ is a topological surface of genus $g$ with $n$
punctures, the mapping class group of $S_{g,n}$ is denoted by
$\Mod(S_{g,n})$, while $\PMod(S_{g,n})$ denotes the pure mapping
class group of $S_{g,n}$, which is defined to be the subgroup of
$\Mod(S_{g,n})$ of elements that fix each punctures individually.
\end{say}

\begin{say}In the sequel we will need the following variants of the
Dehn-Nielsen-Baer Theorem, for which see \cite[Thm. 8.8
p. 234]{fmarga}, \cite[Thm. 5.7.1 p. 197 
and Thm. 5.13.1 p. 214]{zieschang} 
and \cite[\S 2.9]{ivanov}.
\end{say}

\begin{teo}[Dehn-Nielsen-Baer]\label{dehnI}
Let $x=(x_0, X)\in \conf_{0,n+1}S^2$. Then
$\phi\in \Aut^*\pi_1(S^2-X, x_0)$ if and only if there exists
$\sigma \in \Inn \pi_1(S^2-X, x_0)$ and an orientation-preserving
homeomorphism $h: S^2-X\ra S^2-X$ such that $h(x_0)=x_0$ and
$\phi=\sigma \circ h_*$. In other words,
$\Mod(S^2-X)\cong \Out^*(\pi_1(S^2-X, x_0))$.
\end{teo}

\begin{cor}\label{dehnII}
Let $x,y\in \conf_{0,n+1}S^2$ and
$\phi: \pi_1(S^2-X, x_0)\ra \pi_1(S^2-Y, y_0)$ be a homomorphism
that sends geometric bases to geometric bases. Then there exists
$\sigma\in \text{Inn}(\pi_1(S^2-Y, y_0))$ and an
orientation-preserving homeomorphism $h: S^2-X\ra S^2-Y$ such that
$h(x_0)=y_0$ and $\phi=\sigma \circ h_*$.
\end{cor}
\begin{proof}
Fix an orientation preserving homeomorphism $f : S^2-Y \ra S^2-X$
such that $f(y_0) = x_0$ and apply the Dehn-Nielsen-Baer theorem to
$f_*\circ \phi$.
\end{proof}
\begin{cor}\label{dehnIII}
Let $x=(x_0, X)\in \conf_{0,n+1}S^2$. Then
$\phi\in \Aut^{**}\pi_1(S^2-X, x_0)$ if and only if there exists
$\sigma\in \text{Inn}(\pi_1(S^2-X, x_0))$ and an
orientation-preserving self-homeomorphism $h$ of $ S^2$ such that
$h(x_i)=x_i$ for $0\leq i \leq n$ and $\phi=\sigma \circ h_*$. In
other words, $\PMod(S^2-X)\cong \Out^{**} \pi_1(S^2-X, x_0)$.
\end{cor}
\begin{proof}
Applying the Dehn-Nielsen-Baer theorem we get the homeomorphism $h $
of $S^2-X$ and $\sigma$. It is elementary that $h$ extends to a
homeomorphism of $S^2$.  Next assume $h(x_1) = x_j$ and fix a
geometric basis $\base=\{[\alfa_i]\}$ adapted to $x$. Here $\alfa_i$
is a loop at $x_0$ that makes a counterclockwise turn aroung $x_i$
as in \ref{sayhatalfa}. Hence $[h\alfa_1]$ is a loop making a turn
around $h(x_1) = x_j$.  But $[h\alfa_1]$ is conjugate to
$\sigma h_* ([\alfa_1]) = \phi ( [\alfa_1]) $ which is conjugate to
$[\alfa_1]$ since $\phi\in \autss \pi_1(S^2-X, x_0)$. Since
$\alfa_1$ makes a turn around $x_1$ it follows that
$h(x_1)=x_j=x_1$.  Similarly $h(x_i)=x_i$ for any $i$.
\end{proof}

\begin{say} 
We conlude this section interpreting some classical constructions in
the theory of braid groups using parallel transport.  We consider
the (pure version of the) generalized Birman exact sequence
associated with $\C^{**}=\puno-\{0,1,\infty\}$, see \cite[Thm. 9.1,
p. 245]{fmarga}.
\begin{gather}
\label{birman}
1\ra \pi_1(\M_{0,n}, X) \xrightarrow{Push}
\PMod(\puno-X) \xrightarrow{Forget}
\PMod(\C^{**})\ra 1.
\end{gather}
The map $Forget$ is the natural homeomorphism obtained by filling in
the puntures, i.e., it is the map induced by the inclusion
$\puno-X\hookrightarrow\C^{**}$.  The map $Push$ is defined as
follows (see \cite[Section 4.2.1]{fmarga}). Let
$\alpha=(\alpha_1,...,\alpha_{n}): [0,1]\ra \M_{0,n}$ be a pure
braid in $\puno$, with $\alpha(0)=\alpha(1)=X$. Thinking of $\alpha$
as an isotopy from $X$ to $X$ (sending each $x_i$ to $x_i$) we get
by Theorem \ref{Hirsch} that it can be extended to an isotopy of the
whole $\puno$. Denoting by $\Phi_{\alpha}$ the homeomorphism of
$\puno$ obtained at the end of the isotopy, we have that
$\Phi_{\alpha}(x_i)=\alpha_i(1)=x_i$, and thus $\Phi_{\alpha}$ can
be reguarded as an homeomorphism of $\puno-X$. Taking its isotopy
class we get $Push(\alpha)=[\Phi_{\alpha}]\in\PMod(\puno-X)$.  This
map is well defined, i.e., it does not depend on the choice of
$\alpha$ within its homotopy class nor on the choice of the isotopy
extension.
\end{say}

\begin{say}
It is useful to  reinterpret the morpohism $\eps$ defined in \eqref{epsilon} in this setting.
In particular we note
that 
$\Im\varepsilon \subset \Aut^{**}(\pi_1(\puno-X, x_0))$.
Fix $[\alf]\in \pi_1 (\mon, X)$.

Arguing as in  Proposition \ref{esp-parall}  note that
$f^{\alpha}$ extends to a homeomorphism $f^{\alpha}:\puno\ra \puno$
that fixes every $x_i$ individually.
Hence $[f^\alf] \in \PMod (\puno -X)$.
Since $\eps([\alf]) = f_*^\alf$
we have $\eps([\alf]) \in \Aut^{**}(\pi_1(\puno-X))$.

Let
$\tilde\varepsilon: \pi_1(\M_{0,n}, X)\ra \Out^{**}(\pi_1(\C^{**}-X,
x_0)) $ denote the composition of $\varepsilon$ with the natural
projection $\Aut^{**}\ra\Out^{**}$. Also, denote by
$F:\PMod(\C^{**}-X)\ra \Out^{**}(\pi_1(\C^{**}-X, x_0)) $ the
isomorphism $F:[h] \mapsto [h_*]$ coming from Corollary
\ref{dehnIII} of Dehn-Nielsen-Baer Theorem. The following
proposition is the analogue of Theorem 1.10 in \cite{birman} for
configurations of points in $\C^{**}$ (instead of $\C$).
\end{say}

\begin{prop}\label{commut} For
$[\alpha]\in \pi_1(\M_{0,n}, X)$, let $f^{\alpha}$ be the parallel
transport as in Lemma \ref{lemmatrasporto}. Then
$Push([\alpha])=[f^{\alpha}]$. Moreover, the following diagram
commutes
\begin{equation*}
\begin{tikzcd}[row sep=tiny]
&  \PMod(\puno-X)  \arrow[dd, "F"] \\
\pi_1(\M_{0,n}, X) \arrow[ur, "Push"] \arrow[dr, "\tilde\varepsilon"] & \\
& \Out^{**}(\pi_1(\puno-X, x_0))
\end{tikzcd}
\end{equation*}
\end{prop}
\begin{proof}
Let $\alpha: [0,1]\ra \M_{0,n}$ be a pure braid in $\puno$, with
$\alpha(0)=\alpha(1)=X$, that we think as an isotopy from $X$ to
$X$.  Let $\tilde H: (\PP^1 - X) \times [0,1 ] \ra \monu $ and
$f^{\alpha}$ be as in Lemma \ref{lemmatrasporto}.  Define a map
$\psi: \PP^1 \times [0,1]\ra \PP^1$ by
$\psi( u,t): = \tilde H (u,t)$ for $u \not \in X$ and
$\psi(x_i,t):=\alpha_i(t)$.  So $\psi$ is an ambient isotopy of
$\PP^1$ extending the isotopy $\alfa$.  This proves the result,
since by Proposition \ref{esp-parall}
$\varepsilon ([\alfa]) = f^{\alpha}_*$, so
$\tilde\varepsilon([\alfa] ) =f^{\alpha}_* \mod \Inn \pi_1(\PP^1-X,
x_0)$, while $Push ([\alfa] ) = [f^{\alpha}]$.
\end{proof}

\begin{remark}  Considering
configurations of points in $\C$ instead of $\C^{**}$, Proposition
\ref{commut} corresponds to Theorem 1.10 in \cite{birman}.
\end{remark}

\begin{prop}
\label{traspa}
Let $x = (x_0,X) \in \monu$ and let
$\nu \in \Aut^{**}\pi_1(\PP^1-X , x_0)$. Then there are
$[\alpha]\in \pi_1(\mon, X)$, a lifting $\tilde \alpha $ of
$\alpha$ with $\tilde\alpha(0) = \tilde\alpha(1)=x_0$, a parallel
transport $f_t^\alpha $ such that
$f_t^\alpha (x_0) = \tilde\alpha(t)$ and
$\z\in \pi_1(\PP^1-X , x_0)$ such that
$\nu = \inn_\z \circ f^\alpha _*$.
\end{prop}
\begin{proof}
Since $\PMod(\C^{**})$ is trivial (see \cite[Proposition
2.3]{fmarga}), it follows from \eqref{birman} that $Push$ (and
thus $\tilde\varepsilon$) is an isomorphism. In particular, for
every ${\nu}\in \Aut^{**}(\pi_1(\puno-X, x_0))$, there exists
$[\alpha]\in \pi_1(\M_{0,n}, X)$ and
$\sigma\in \Inn(\pi_1(\puno-X, x_0))$ such that
$f^{\alpha}_*=\varepsilon([\alpha])={\nu}\circ\sigma$. Thus
$\nu = \inn_\z \circ f^\alpha _*$ for some
$\z\in \pi_1(\PP^1-X , x_0)$.
\end{proof}

\section{Topological types of actions}\label{sec:tipitopo}

\begin{defin}Let $G$ be a finite group and let $\Sigma_1$ and
$\Sigma_2$ be oriented topological surfaces both endowed with an
action of $G$. We say that the two actions are \emph{topologically
equivalent} if there are $\eta \in \Aut G$ and an orientation
preserving homeomorphism $f: \Sigma_1 \cong \Sigma_2$ such that
$f(g\cd x) = \eta (g) \cd f(x)$ for any $x\in \Sigma_1$ and any
$g\in G$, see \cite{ganzdiez}.  An equivalence class is called a
\emph{topological type} of $G$-action (sometimes this is called
\emph{unmarked topological type}).
\end{defin}

\begin{defin}\label{tipitopo}
Fix on $S^2$ the orientation by the outer normal. We let
$\topo^n(G)$ denote the set of topological types of $G$-actions on a
topological surface $\Sigma$ such that $\Sigma/G\cong S^2$ (as
oriented surfaces) and the projection $\pi: \Sigma \ra \Sigma/G$ has
$n$ branch points.
\end{defin}


\begin{defin}
\label{def:data}
If $G$ is a finite group an $n$-datum is an epimorphism
$\theta : \Ga_n \ra G$ is such that $\theta(\ga_i)\neq 1$ for
$i=1, \lds, n$. We let $\datan(G)$ denote the set of all $n$-data
associated with the group $G$.
\end{defin}

\begin{say} Fix a point $x=(x_0, X)\in \conf_{0,n+1}S^2$ and a
geometric basis $\mathscr{B}=\{[\alfa_i]\}_{i=1}^n$ of
$\pi_1(S^2\setminus X, x_0)$.  Denote by
$\chi: \Ga_n \cong \pi_1(S^2 - X, x_0)$ the corresponding
isomorphism. If $\theta: \Gamma_n \ra G$ is a $n$-datum, the
epimorphism $\theta\circ \chi\meno$ gives rise to a topological
$G$-covering $p:\Sigma_0^{\theta} \ra S^2 -X$.  By the topological
part of Riemann Existence Theorem this can be completed to a
branched $G$-cover $p: \Sigma^{\theta} \ra S^2$. By taking the
equivalence class of $\Sigma^{\theta}$ we get a topological type of
$G$-action. We get a map
\begin{gather*}
\mathscr{F}_{x, \mathscr{B}}: \datan(G)\ra \mathscr{T}^n(G), \quad
(\theta:\Gamma_n\ra G)\longmapsto [\Sigma^{\theta}]
\end{gather*}
\end{say}


\begin{say}
\label{Hurwitz} 
We now introduce an action on the set of data that will be very important for the rest of the paper. 
By Dehn-Nielsen-Baer Theorem
$\outa \cong \Out^*(\pi_1(S^2-X, x_0)) \cong \Mod(S^2-X)$.  The
latter group has a presentation with generators
$\sigma, \lds, \sigma_{n-1} $ and relations
\begin{equation}
\label{umpf}
\begin{gathered}
\sigma_i \sigma_j =
\sigma_j \sigma_i \quad \text{for} \, \, |i-j|\geq 2, \qquad
\sigma_{i+1}\sigma_i\sigma_{i+1}=\sigma_i\sigma_{i+1}\sigma_i, \\
\sigma_1 \cds \sigma_{n-2} \sigma_{n-1}^2 \sigma_{n-2} \cds \sigma_1 = 1,\qquad
( \sigma_1 \cds \sigma_{n-1} )^n =1.
\end{gathered}
\end{equation}
(See \cite[Thm. 4.5 p. 164]{birman}.)  Let
$\ts_i: \Ga_r \ra \Ga_r$ be the automorphism defined by the rule:
\begin{gather*}
\ts_i (\gamma_i) = \gamma_{i+1}, \quad \ts_i
(\gamma_{i+1}) = \gamma_{i+1} ^{-1} \gamma_i \gamma_{i+1}, \quad
\ts_i (\gamma_j ) = \gamma_j \quad \text{for }j \neq i, i+1.
\end{gather*}
These automorphisms belong to $\auta$ and satisfy the relations
\eqref{umpf}. Therefore there is a (unique) morphism
$\phi : \outa \ra \auta$ such that $\phi(\sigma_i) : = \ts_i$. {The group $\Aut^*\Gamma_n \times \aug $ acts on the set $\datan(G) $ by the rule
\begin{gather*}
(\nu, \eta) \cd \theta : = \eta \circ \theta \circ \nu\meno,
\end{gather*}
where $(\nu, \eta) \in \Aut^*\Gamma_n\times \Aut G$ and $\theta\in \datan(G)$ is a datum.}
Using $\phi$ we let $\Aut G \times \outa$ act on $\datan$: if
$(\eta,\sigma) \in \Aut G \times \outa$ and $\theta \in \datan$, then
\begin{gather*}
(\eta,\sigma) \cd \theta :=  \eta \circ \theta \circ \phi(\sigma)\meno.
\end{gather*}
We claim that the actions of $\Aut G \times \auta $ and of $\Aut G\times \outa$ on $\datan$ have the same orbits, hence
\begin{gather*}
\datan / \Aut G \times \auta = \datan / \Aut G \times \outa.
\end{gather*}
The reason is that $\phi$ is a splitting of the short exact sequence
\begin{gather*}
1 \lra \Inn  \Ga_n \hookrightarrow \auta \stackrel{\pi}{\lra} \outa \ra 1,
\end{gather*}
i.e.  $\pi\circ \phi$ is the identity on $\outa$.  So 
every $\nu \in \auta$ can be factored as
$
\nu = \phi(\sigma) \cd \inn_a,
$ with $a\in \Ga_n$ and
\begin{gather*}
(\eta, \nu) \cd \theta = \eta \circ \theta \circ( \inn_a)\meno \circ \phi(\sigma)\meno = \eta \circ \inn_{\theta(a)\meno} \circ \theta  \circ \phi(\sigma)\meno = \\
=( \eta \circ \inn_{\theta(a)\meno}, \sigma ) \cd \theta.
\end{gather*}
This shows that the orbits of the two actions coincide.
\end{say}


\begin{teo}\label{teotipitopo}
Let $G$ be a finite group. Choose
\begin{enumerate}
\item an element $x=(x_0, X)\in \conf_{0,n+1}S^2$;
\item a geometric basis $\mathscr{B}=\{[\alfa_i]\}_{i=1}^n$ of
$\pi_1(S^2\setminus X, x_0)$.
\end{enumerate}
Then the map $\mathscr{F}_{x, \mathscr{B}}$ induces a bijection
between $\datan(G)/(\auta \times \aug)$ and the set
$\mathscr{T}^n(G)$ of topological types of $G$-actions. The
bijection does not depend on the choices of the point
$x\in \conf_{0,n+1}S^2$ and of the geometric basis $\mathscr{B}$.
\end{teo}
From the discussion in \ref{Hurwitz} we immediately get the following.
\begin{cor}\label{outout}
The topological types of $G$-actions on curves of genus $g$ are in bijection with $\datan/\Aut G\times \outa$.
\end{cor}
The proof of Theorem \ref{teotipitopo} is based on the following two
Propositions.

\begin{prop}\label{lem3}
The map $\mathscr{F}_{x, \mathscr{B}}$ is constant on the orbits of
the action of $\auta \times \aug $.
\end{prop}
\begin{proof}
Let $\theta: \Gamma_n\ra G$ be a datum and
$(\nu, \eta)\in \auta\times \Aut G$. Let
$\theta'= \eta \circ \theta \circ \nu^{-1}$. We want to show that
$\Sigma^{\theta}$ and $\Sigma^{\theta'}$ have the same topological
type of $G-$action. Set
$\overline{\nu}=\chi\circ \nu \circ \chi^{-1}$. Observe that
$\overline{\nu}\in \Aut^*(\pi_1(S^2\setminus X, x_0))$ since
$\nu\in \auta$.  By the Dehn-Nielsen-Baer Theorem \ref{dehnI}, there
is $\sigma\in \text{Inn}(\pi_1(S^2-X, x_0))$ and an
orientation-preserving diffeomorphism
$h : (S^2 -X , x_0) \ra (S^2 -X , x_0) $ such that $h(x_0)=x_0$ and
$\sigma \circ
h_*=\overline{\nu}$. 
Let $p, p'$ denote the
projections:
\begin{equation*}
\begin{tikzcd}
\Sigma^{\theta}_0 \arrow[d, "p"]
& \Sigma^{\theta'}_0\arrow[d, "p'"] \\
(S^2-X, x_0)\arrow[r,"h" ] & (S^2-X, x_0).
\end{tikzcd}
\end{equation*}
Choose $\tilde{x}_0 \in \Sigma_0^\theta $ and
$\tilde{x}'_0 \in \Sigma_0^{\theta'} $ both over $x_0$.  We have
that
$h_*(p_*(\pi_1(\Sigma^\theta_0,\tilde{x}_0)))=(\sigma^{-1}\circ
\overline{\nu})(\ker(\theta\circ
\chi^{-1}))=\sigma^{-1}(\ker(\theta\circ \chi^{-1}\circ
\overline{\nu}^{-1}))=\ker(\theta\circ \chi^{-1}\circ
\overline{\nu}^{-1})$, where the last equality holds because
$\sigma$ is an inner automorphism. Morever, since $\eta\in \Aut G$,
$\ker(\theta\circ \chi^{-1}\circ
\overline{\nu}^{-1})=\ker(\eta\circ\theta\circ \chi^{-1}\circ
\overline{\nu}^{-1})$. Thus
$h_*(p_*(\pi_1(\Sigma^\theta_0,\tilde{x}_0)))=\ker(\eta\circ\theta\circ
\chi^{-1}\circ
\overline{\nu}^{-1})=(p')_*(\pi_1(\Sigma^{\theta'}_0))$.  By the
lifting theorem we get an oriented homeomorphism
$\tilde{h}: \Sigma^{\theta}_0\ra \Sigma^{\theta'}_0$, such that the
diagram commutes and which extends to the compactifications. Hence
the $G$-actions on $\Sigma^{\theta}$ and $\Sigma^{\theta'}$ have the
same topological type.
\end{proof}

\begin{prop}\label{lem4}
If $\theta\in\datan(G)$, then $\mathscr{F}_{x, \mathscr{B}}(\theta)$
does not depend on the choices of the point $x\in \conf_{0,n+1}S^2$
and of the geometric basis $\mathscr{B}$.
\end{prop}
\begin{proof}
First fix $x$ and consider two geometric bases $\mathscr{B}$ and
$\overline{\mathscr{B}}$. Let
$\chi, \bar{\chi} : \Ga_n \ra \pi_1(S^2-X, x_0)$ denote the
corresponding isomorphisms. Then
$\nu:= \chi\meno \circ \bar{\chi} \in \Aut^*\Ga_n $.  For a datum
$\theta$, we have
$\theta \circ \bar{\chi}\meno = \theta \circ \nu\meno \circ
\chi\meno$. So
$\mathscr{F}_{x, \overline{\mathscr{B}} } (\theta) =
\mathscr{F}_{x,\mathscr{B}} (\theta \circ \nu\meno)$. By Proposition
\ref{lem3},
$\mathscr{F}_{x,\mathscr{B}} (\theta \circ
\nu\meno)=\mathscr{F}_{x,\mathscr{B}} (\theta )$. Hence
$\mathscr{F}_{x, \overline{\mathscr{B}}} (\theta) =
\mathscr{F}_{x,\mathscr{B}} (\theta )$ as desired.
Now suppose that $x, y \in \conf _{0,n+1} S^2$. Let
$\chi : \Ga_n \ra \pi_1(S^2 -X, x_0) $ and
$\bar{\chi} : \Ga_n\ra \pi_1(S^2-Y,y_0) $ be the isomorphisms
associated with two geometric bases $\base$ and
$\overline{\base}$. Then
$\nu: = \bar{\chi} \circ \chi\meno : \pi_1(S^2 -X, x_0) \ra
\pi_1(S^2-Y,y_0) $ sends a geometric basis to a geometric
basis. Hence by Corollary \ref{dehnII}, there is
$\sigma\in \text{Inn}(\pi_1(S^2-Y, y_0)))$ and an orientation
preserving homeomorphism $h: (S^2 -X, x_0) \ra (S^2 -Y, y_0)$ such
that $h(x_0)=y_0$ and $\sigma \circ h_* = \nu$.  Given a datum
$\theta$, $h_*$ maps the kernel of $\theta\circ \chi\meno$ to the
kernel of $\theta \circ \bar{\chi} \meno$. By the lifting theorem
there is an oriented diffeomorphism $\tilde{h}$ that extends to the
compactifications. Hence the $G$-actions on $\Sigma^\theta_{x}$ and
$\Sigma^\theta_{y}$ have the same topological type.
\end{proof}

We recall two 
basic facts about monodromy maps.  Let $p:E\ra B$
be a topological $G$-covering. For $b\in B$ and $e\in p^{-1}(b)$, we
denote by $\mu_{p,e}$ the monodromy map $\mu_{p,e}:\pi_1(B, b)\ra G$
such that $g=\mu_{p,e}[\alpha]$ maps $e$ to $\alpha_e(1)$, where
$\alpha_e$ is the lifting of $\alpha$ with initial point $e$.

\begin{lemma}\label{lem1}
Let $p:E\ra B$ be a topological $G$-covering. Fix $b_0, b_1\in B$
and $e_i\in p^{-1}(b_i)$. Let $\delta$ be a path from $e_0$ to $e_1$
and $\gamma=p\circ \delta$. Then
$\mu_{p,e_0}=\mu_{p,e_1}\circ \gamma_{\#}$. In particular, if
$b_0=b_1$, $\mu_{p,e_0}$ and $\mu_{p,e_1}$ differ by an inner
automorphism of $\pi_1(B,b_0)$ or - equivalently - of $G$.
\end{lemma}

\begin{lemma}\label{lem2}
Let $p: E \ra B$ and $p': E' \ra B'$ be $G$-coverings. Let
$\tilde h: E\ra E'$ be a $G-$equivariant homeomorphism and denote by
$h: B\ra B'$ the induced homeomorphism. Fix $e_0\in E$. Then
$\mu_{p,e_0}=\mu_{p', \tilde{h}(e_0)} \circ h_*$.
\end{lemma}

\begin{proof}[Proof of Theorem \ref{teotipitopo}] By Proposition
\ref{lem3}, $\mathscr{F}_{x, \mathscr{B}}$ induces a map between
$\datan(G)/$ $(\auta \times \aug)$ and $\mathscr{T}^n(G)$. To prove the
statement we have to check the following:
\begin{enumerate}
\item if two epimorphisms $\theta, \theta': \Gamma_r \ra G$ give
rise to the same topological type of $G$-action, then $\theta$ and
$\theta'$ are in the same orbit for the action of
$\auta \times \aug $;
\item every topological type of $G$-action with $n$ branch points
can be constructed from a datum in $\datan(G)$.
\end{enumerate}
To prove (1), consider the branched covers $p: \Sigma \ra S^2$ and
$p': \Sigma' \ra S^2$ associated with $\theta\circ \chi\meno$ and
$\theta'\circ \chi\meno$ and suppose that there exists
$\eta\in \Aut G$ and an orientation preserving homeomorphism
$\tilde h:\Sigma\ra \Sigma'$ such that
$\tilde h (g\cd e) = \eta(g) \tilde h (e)$.  We get an induced
homeomorphism $h: \Sigma/G\ra \Sigma'/G$ and an isomorphism
$ h _*: \pi_1(S^2\setminus X, x_0)\ra \pi_1(S^2\setminus X,
h(x_0))$. Fix $e_0\in p^{-1}(x_0)$. From Lemma \ref{lem2} it follows
that $\mu_{p,e_0}=\eta\circ \mu_{p', \tilde h(e_0)} \circ h_*$. Now
fix $e_0'\in (p')^{-1}(x_0)$ and a path in $\Sigma'$ from
$\tilde h(e_0)$ to $e'_0$. Finally let $\gamma=p'\circ \delta$. By
Lemma \ref{lem1} we get that
$\mu_{p',\tilde h (e_0)}=\mu_{p',e_0'}\circ \gamma_{\#}$. Thus
\begin{gather}\label{monod}
\mu_{p,e_0}=\eta \circ \mu_{p',e_0'}\circ \gamma_{\#}\circ h_*.
\end{gather}
Observe that, since $\tilde h$ preserves the orientation, so does
$h$, hence
$\gamma_{\#}\circ h_*: \pi_1(S^2\setminus X, x_0)\ra
\pi_1(S^2\setminus X, x_0)$ lies in
$\Aut^*(\pi_1(S^2\setminus X, x_0))$. Let
$\nu:=\chi^{-1}\circ (\gamma_{\#}\circ h_*)\circ \chi\in \auta$ be
the corresponding automorphism in $\auta$. (Again we are using that
$\chi$ comes from a geometric basis.) Also, observe that
$\theta\circ \chi^{-1}$ coincides with $\mu_{p,e_0}$ up to an inner
automorphism of $G$, and the same holds for $\theta'\circ \chi^{-1}$
and for $\mu_{p',e'_0}$. We get that there exists $\eta\in \Aut G$
such that \eqref{monod} becomes
\begin{gather*}
\theta\circ \chi^{-1}=\eta \circ \theta'\circ \chi^{-1} \circ (
\gamma_{\#}\circ h_*)=\eta \circ \theta'\circ \nu^{-1}\circ
\chi^{-1}.
\end{gather*}
Thus $(\eta, \nu). \theta'= \theta$; that is, they are in the same orbit for the action of $\auta \times \aug $. To prove (2) assume that $G$ acts effectively on a surface $\Sigma$
in such a way that $\Sigma/G \cong S^2$. Up to diffeomorphism we can
assume that the set of critical values of $p : \Sigma \ra S^2$
coincides with $X$.  Fix a point $\tilde{x}_0 \in p\meno(x_0)$. Let
$\theta:=\mu_{p, \tilde{x}_0} \circ \chi: \Gamma_n\ra G$ be the
monodromy of the unramified cover. Since $\Sigma^{\theta}_0$ is
connected ${\theta}$ is surjective, and $\theta(\gamma_i)\neq 1$
since all the points of $X$ are branch points.  So it is an
$n$-datum. By construction the associated cover
coincides with $\Sigma$. Finally, it follows from Proposition \ref{lem4} that the bijection
induced by $\mathscr{F}_{x, \mathscr{B}}$ does not depend on $x$ and
$\mathscr{B}$.

\end{proof}


\section{Tools for the construction}\label{sec:groups}
This section is dedicated to some tools that we will need in the following section for the construction of the families. We start with some considerations from group theory, that will be at the basis of the construction of the ordered set $(\coll,\geq)$ of Theorem \ref{mainz}.\\

Consider an exact sequence of groups
\begin{gather*}
(*) \qquad 1 \ra N \stackrel{i}{\lra} K \stackrel{p}{\lra} H\ra 1
\end{gather*}
\label{pagestar} and an epimorphism
\begin{gather*}
f : N \twoheadrightarrow G
\end{gather*}
onto a \emph{finite} group $G$.
\begin{defin}
An \emph{extension} $a$ of $(*,f)$ is a pair $a= (H_a,f_a)$, whose
first element is a subgroup $ H_a $ of $ H$ of \emph{finite index};
and whose second element is a morphism $f_a: p\meno(H_a) \ra G$ such
that $f_a i = f$.  We denote by $ \coll (*, f)$ the set of all
extensions.
\end{defin}
If $a=(H_a, f_a)$ is an extension we set
\begin{equation*}
K_a:=p\meno(H_a) .
\end{equation*}
$K_a$ is a subgroup of $K$ and $f_a $ is defined on $K_a$.\\
On the set $\collf$ we introduce the order relation
\begin{gather*}
a \geq b \Longleftrightarrow H_a \subset H_b \text{ and }
f_a=f_b\restr{K_a}.
\end{gather*}

\begin{prop}
$(\coll(*,f), \geq )$ is a directed set.
\end{prop}
\begin{proof}
Given $a,b \in \coll(*,f)$, set
$H_c :=\{h\in H_a\cap H_b: f_a(h)=f_b(h)\}$.  Then $H_c$ has finite
index in $H$ since $G$ is finite. Set $f_c:= f_a\restr{H_c}$. Then
$c:=(H_c, f_c) \in \collf$ and $c \geq a, c \geq b$.
\end{proof}

In the following lemmas we describe two natural bijections between the
sets $\collf$, when $f$ and $(*)$ change under some specific rule.

\begin{lemma}\label{azioneautG}
Given $f: N \tra G$ and $\eta\in \Aut G$, set
$\bar f:= \eta \circ f$. Then
\begin{gather}
\label{rule2}
\Phi: \collf \ra \coll({*}, \bar{f}), \qquad \Phi(H_a,f_a):= (H_a,
\eta \circ f_a ).
\end{gather}
is an order preserving bijection.
\end{lemma}
The proof is immediate.

\begin{lemma}
\label{ses-isomorfe}
Consider a commutative diagram of groups
\begin{equation*}
\begin{tikzcd}[column sep=small, row sep=scriptsize]
(*) & 1 \arrow{r}& N \arrow{d}{\alfa} \arrow{rr}{i}
& & K \arrow{rr}{p} \arrow{d}{\gamma}{} && H \arrow{d}{\beta} \arrow{r} & 1\\
(\bar{*}) & 1\arrow{r} & \bar{N} \arrow{rr}{\bar{i}} & &\bar{K}
\arrow{rr}{\bar{p}} &&\bar{H} \arrow{r}& 1
\end{tikzcd}
\end{equation*}
with exact rows and $\alfa, \beta , \ga$ isomorphisms. In other
words $(*)$ and $(\bar{*})$ are isomorphic short exact
sequences. Given $\bar{f}: \bar{N} \twoheadrightarrow G$ set
$f:=\bar f\circ \alfa : N \tra G$.  Then the map
\begin{gather}
\label{rule}
\Phi: \collf \ra \coll(\bar{*}, \bar{f}), \qquad \Phi(H_a,f_a):=
(\beta (H_a), f_a \circ \gamma\meno \restr{\gamma (K_a)}).
\end{gather}
is an order preserving bijection.
\end{lemma}
\begin{proof}
If $a=(H_a,f_a)$, set $K_a=p\meno(H_a)$ as above.  Set
$\bar{H}_\bara := \beta (H_a)$.  Then
$\bar{K}_\bara: = \bar{p}\meno (\beta (H_a)) = (\beta \meno
\bar{p})\meno (H_a) = (p\ga\meno)\meno(H_a)= \ga(K_a)$.  Set also
$\bar{f}_\bara := f_a \circ \ga\meno\restr{\bar K_\bara} $.  Then it
is immediate to check that
$\bara := (\bar{H}_\bara, \bar{f}_\bara) = \Phi(a)$ belongs to
$\coll(\bar{*}, \bar{f})$ and that $\Phi$ is an order preserving
bijection.
\end{proof}


\begin{lemma}\label{lemmagruppi}
Let $N$, $H$ and $G$ be groups and let
$\dire : H \ra \Aut N, h \mapsto \dire_h$ be a morphism.
Let $f: N \ra G$ and $\morp: H \ra G$ be morphisms. There is a
morphism $\morff: N\rtimes_\eps H \ra G$ extending both $f$ and
$\phi$ (when $N$ and $H$ are included in $N\rtimes_\eps H$ in the
obvious way) if and only if for any $h\in H$
\begin{gather}
\label{exmor}
\inn_{\morp(h)} \circ \morf = \morf\circ \dire_h.
\end{gather}
\end{lemma}
The proof is elementary.

\begin{lemma}\label{grupplemma}
Let $N$, $H$, $G$, $\dire : H \ra \Aut N$ and $\morf$ be as
above. Assume that $\morf$ is surjective, that $N$ is finitely
generated and that $G$ is finite. Then (a)
$ H'':=\{h\in H: \dire_h (\ker \theta) = \ker \theta\}$ is a finite
index subgroup of $H$; (b) there is a morphism
$\tdire: H''\ra \Aut G$ such that the diagram
\begin{equation*}
\begin{tikzcd}[row sep=scriptsize]
N \arrow{r}{\dire_ h} \arrow[swap]{d}{\morf} &
N \arrow{d}{\morf}      \\
G \arrow{r}{\tdire_ h} & G
\end{tikzcd}
\end{equation*}
commutes for $h\in H''$; (c) $H':=\ker \tdire $ is a finite index
subgroup of $H$ and (d) there is a unique morphism
$\morff : N\rtimes_\eps H' \ra G$ that extends $\morf$ and such that
$\morff \restr{H'}\equiv 1$.
\end{lemma}
\begin{proof}
The subgroup $\ker \morf$ has index $d:=|G| < \infty$ in $N$. Since
$N$ is finitely generated, there are a finite number of index $d$
subgroups of $N$ (see e.g. \cite[p. 56]{kurosh-2} or
\cite[p. 128]{hall}.)  The natural action of $\Aut N$ on the
subgroups of $N$ preserves the index. Therefore the orbit of
$\Aut N$ through $\ker \morf$ is finite.  Hence
$(\Aut N)_{\ker \morf}$ has finite index in $\Aut N$. Since
$H / H'' $ injects in $\Aut N / (\Aut N)_{\ker \morf}$, also $H''$
has finite index in $H$. The existence of $\tdire_h$ follows
immediately from the inclusion $\dire_h (\ker f) \subset \ker f$ for
$h\in H''$. Since $\Aut G$ is finite $H'$ has finite index in $H''$
and in $H$. By construction for any $h\in H'$ we have
$\morf = \tdire_h \circ \morf = \morf \circ \dire_h$, i.e. \eqref
{exmor} holds with $\phi : H' \ra G$ the trivial morphism.
\end{proof}

\begin{teo}
If the sequence $(*)$ splits, then $\collf \neq \vacuo$.
\end{teo}
\begin{proof}
By Lemma \ref{ses-isomorfe} we can assume that the split exact
sequence $(*)$ is a semidirect product. The result then follows from
Lemma \ref{grupplemma}.
\end{proof}

\begin{say}
\label{ex-pippone}
We dedicate the second part of this section to some considerations on
coverings and fiber bundles, which will be fundamental tools for our
construction.

In the following we assume that all the spaces considered are
semilocally 1-connected.  Let $X$ be a connected space and let
$x\in X$. For every subgroup $H\subset \pi_1(X,x)$ there is a
pointed covering $p: (E,e) \ra (X,x)$ such that $\im
p_*=H$. Moreover $p$ is unique up to pointed isomorphism.  If
$\beta \in \Omega (X,x,x')$ and
$\beta_\#: \pi_1(X,x)\ra \pi_1(X,x')$ is the induced isomorphism,
then the pointed coverings of $X$ associated with
$H\subset \pi_1(X,x)$ and with $\beta_\# H \subset \pi_1(X,x')$ are
isomorphic.  Indeed if $\beta_e$ denotes the lifting with
$\beta_e(0) =e$ and $e'=\beta_e(1)$, then
$p_* \pi_1(E,e')=\beta_\# H$, so $E$ is associated with both
subgroups.  If $X$ is a complex manifold any covering has a unique
complex structure such that $p$ is holomorphic and the coverings
associated to $H\subset \pi_1(X,x)$ and with
$\beta_\# H \subset \pi_1(X,x')$ are biholomorphic.
\end{say}

\begin{lemma}\label{fibratirivestimenti}
Let $\bar E$, $\bar B$ and $B$ connected and locally arcwise
connected topological spaces.  Let $p:\bar E\ra \bar B$ be a fibre
bundle and let $q: B \ra \bar B$ be a covering.  Let $E:=q^*\bar E$
be the pull-back bundle.  Consider the diagram:
\begin{equation}\label{doppiafibra}
\begin{tikzcd}
(E,e) \arrow[r, "\bar q"] \arrow[d, swap, "\burp"]
& (\bar E, \bar {e})\arrow[d, "p" ] \\
(B,b) \arrow[r, "q"] & (\bar B,\bar{b}).
\end{tikzcd}
\end{equation}
Then $\bar {q}: E \ra \bar E$ is also a covering.  Moreover if the
fibre of $p$ is arcwise connected, {then}
$\bar{q} _* \pi_1(E, e) = p_*\meno ( q_* \pi_1(B,b) )$.
\end{lemma}
\begin{proof}
Fix $ \bar e\in \bar E$, set $ \bar b=p( \bar e)$ and let
$V\subset \bar B$ be an evenly covered open subset of $ \bar B$,
i.e. $ q\meno (V) = \bigsqcup U_i$ and $q\restr{U_i} $ is a
homeomorphism of $U_i $ onto $V$. We claim that $p\meno(V)$ is an
evenly covered neighbourhood of {$e$}. Indeed
$\bar{q}\meno (p\meno(V)) = \bigsqcup {\burp }\meno (U_i)$. Moreover
$ \burp\meno (U_i) = (q\restr{U_i})^* {\overline{E}} $ is mapped
homeomorphically on $p\meno (V)$ by $\bar{q}$ since $q\restr{U_i}$
is a homeomorphism onto $V$. This proves the first assertion.  Next
choose $e \in \bar{q}\meno({\overline{e}})$ and set
$b =\burp ( e) $. Obviously $q(b) = \bar{b}$.  Set
$\bar F:=p\meno( \bar b)$, $ F := \burp \meno (\bar{b})$.  The
diagram \eqref{doppiafibra} induces a morphism of the homotopy exact
sequences of the bundles:
\begin{equation*}
\begin{tikzcd}
\arrow {r} & \pi_2 ({B}) \arrow{d}{\cong} \arrow {r} &
\pi_1({F}, {e}) \arrow{d}{\cong} \arrow{r}& \pi_1( E, {e})
\arrow{d}{\bar q_*} \arrow{r}{p_*} &
\pi_1(  B, {b}) \arrow {d}{q_*} \arrow{r} & \pi_0 (F) =1 \arrow{d}{\cong}  \\
\arrow {r} & \pi_2 ( \bar B) \arrow {r} & \pi_1( \bar F, \bar e)
\arrow{r}& \pi_1( \bar E, \bar e) \arrow{r}{p_*} & \pi_1( \bar
B, \bar b) \arrow{r} & \pi_0 (F) =1 .
\end{tikzcd}
\end{equation*}
Set $ H:=q_* \pi_1(B,b) \subset \pi_1(\bar B,\bar b)$, and
$K:= p_*\meno (H) \subset \pi_1(\bar E,\bar e) $.  In the lower row
we can substitute $\pi_1 (B,b)$ with $H $ and $\pi_1(E,e)$ with $K$
and the row remains exact. Clearly $\bar q_* $ maps into $K$ since
the diagram commutes. So we get the diagram
\begin{equation*}
\begin{tikzcd}
\arrow {r} & \pi_2 ({B}) \arrow{d}{\cong} \arrow {r} &
\pi_1({F}, {e}) \arrow{d}{\cong} \arrow{r}& \pi_1( E, {e})
\arrow{d}{\bar q_*} \arrow{r}{p_*} &
\pi_1(  B, {b}) \arrow {d}{q_*} \arrow{r} & \pi_0 (F) =1 \arrow{d}{\cong}  \\
\arrow {r} & \pi_2 ( \bar B) \arrow {r} & \pi_1( \bar F, \bar e)
\arrow{r}& K \arrow{r}{p_*} & H \arrow{r} & \pi_0 (F) =1 .
\end{tikzcd}
\end{equation*}
Now $q_*$ is an isomorphism.  Applying the short five lemma
\cite[p. 16]{eilesteen}, we get that $K = \im \bar q_* $ as desired.
\end{proof}

The following lemma is a sort of converse which will be needed later.
\begin{lemma}
\label{converse}
Let $A, \bar E, E, B, \bar B$ be connected and locally arcwise
connected topological spaces.  Consider the diagram:
\begin{equation}\label{AEBB}
\begin{tikzcd}[column sep = small,row sep = scriptsize]
(A,a) \arrow[r, "\tilde q"] \arrow[d, swap, "\phi"]
& (\bar E, \bar {e})\arrow[d, "p" ] \\
(B,b) \arrow[r, "q"] & (\bar B,\bar{b}).
\end{tikzcd}
\end{equation}
Assume (a) that $\phi: A \ra B$ and $p: \bar E \ra \bar B$ are fibre
bundles with arcwise connected fibres, (b) that $q$ and $\tilde q$
are finite degree coverings, (c) that
$ \tilde q _* \pi_1(A,a) = p_*\meno (q_* \pi_1(B,b))$. Then $ A $ is
isomorphic to $ q^* \bar E$ as a fibre bundle over $B$.
\end{lemma}
\begin{proof}
Apply Lemma \ref {fibratirivestimenti}. Using the same notation as
in \eqref {doppiafibra} we obtain
$\bar q_* \pi_1 (E,e) = p_*\meno (q_* \pi_1(B,b)) = \tilde q _*
\pi_1(A,a)$. Moreover $\bar q$ is also a covering.  So there is
$w: (A,a) \ra (E,e)$ such that $\bar q \circ w = \tilde q$.  It
remains to show that $\psi \circ w = \phi$.  Combining \eqref{AEBB}
with \eqref {doppiafibra} we get the commutative diagram
\begin{equation*}
\begin{tikzcd}[column sep = small,row sep = scriptsize]
(A,a) \arrow[bend left]{rr}{\tilde{q}} \arrow{r}{w} \arrow[bend
right] {dr}[swap]{\phi} & (E,e) \arrow[r, "\bar q"] \arrow[d,
swap, "\burp"]
& (\bar E, \bar {e})\arrow[d, "p" ] \\
& (B,b) \arrow[r, "q"] & (\bar B,\bar{b}).
\end{tikzcd}
\end{equation*}
From $\bar q \circ w = \tilde q$ we get
$p \circ \bar q \circ w =p \circ \tilde q$, hence
${q} \circ \burp \circ w = {q} \circ \phi$. So $\burp \circ w$ and
$\phi$ lift the same map with respect to the covering $q$. Since
$\burp \circ w (a) = \phi (a)$, we conclude that
$\burp \circ w = \phi$ and the result follows.
\end{proof}

\section{Construction of the families \texorpdfstring{of $G$-curves}{}
}
\label{sec:famG-curve}

\begin{say}
Fix an element $x=(x_0, X)\in \M_{0,n+1}$ and set
\begin{gather}
\label{notax} N_x:=\pi_1(\PP^1 -X, x_0) \quad K_x:= \pi_1(\monu,
x) \quad H_X:=\pi_1 (\mon, X).
\end{gather}
Consider the split exact sequence in the top row of \eqref{seqpuni},
namely:
\begin{equation*}
\begin{tikzcd}[column sep=small]
(*_x) & 1 \arrow{r}& N_x \arrow{rr}{i_*} & & K_x \arrow{rr}{p_*}
&& H_X \arrow{r} & 1.
\end{tikzcd}
\end{equation*}
Here $i: \PP^1-X \hookrightarrow \monu$ is the map
$i(x'):=(x',x_1, \lds, x_n) $ and $p: \monu \ra \mon$ is the
fibration. Now let $G$ be a finite group and let
$\theta: \Gamma_n\ra G$ be a datum. Choose a geometric basis
$\mathscr{B}=\{[\alfa_i]\}_{i=1}^n$ of $N_x$. As in \ref{bbbendef},
let $\chi: \Ga_n \ra N_x$ be the isomorphism induced from the basis
$\base$. We apply the group theoretical considerations of Section
\ref{sec:groups} to the exact sequence $(*_x)$ with
$f:= \theta \circ \chi\meno : N_x \tra G$. We get a directed set
$\coll(*_x, f)$, which is nonempty since $(*_x)$ splits.  To stress
the dependence from the choices made, we will set
\begin{gather*}
\coll(x,\base, \theta) := \coll(*_x, \theta\circ \chi\meno).
\end{gather*}
Indeed $\chi$ contains the same information as the basis $\base$.
\end{say}

\begin{defin}\label{colldef}
A \emph{collection of families} is an indexed set
$\{ \cc_a \ra \basef_a\}_{a\in \coll}$ where
\begin{enumerate}
\item $(\coll, \geq)$ is a directed set;
\item $(\basef_a, y_a)$ is a pointed smooth complex quasiprojective
variety;
\item $ \cc_a \ra \basef_a$ is a family of curves;
\item if $a, b\in \coll$, $a\geq b$, then there is an \'etale cover
of finite degree $v_{ab}: (\basef_ a,y_a) \ra ( \basef _b, y_b) $
such that $\cc_a \cong v_{ab}^*\cc_b$.
\end{enumerate}
\end{defin}
In this section we construct a collection of families indexed by
$\coll(x,\base, \theta)$.

\begin{say}
Fix $\ax=(H_a,f_a) \in \coll(x,\base, \theta)$. Let
$q_\ax: (\basa,\yy_\ax) \ra (\mon,X)$ be the pointed covering with
$q_{a*} \pi_1(\basa , x_a) = H_a$.  Endow $\basa$ with the unique
structure of a complex manifold making $q_a$ an unramified analytic
cover.  Consider the diagram
\begin{equation}
\label{Eax}
\begin{tikzcd}
(E_\ax: = q_\ax^* \monu, e_a) \arrow{r}{\bar q_\ax}
\arrow{d}{\burp_\ax}
&   (\monu, x) \arrow{d}{p} \\
(\basa, y_a) \arrow{r}{q_\ax} & ( \mon, X),
\end{tikzcd}
\end{equation}
with $e_a:=(y_a, x)$.  Notice that $p : \monu \ra \mon$ is the
universal family of lines with $n$ holes and hence
$\burp_\ax : E_\ax \ra \basa$ is also a holomorphic family of curves
(lines with holes).
\end{say}

By Lemma \ref {fibratirivestimenti} applied to the diagram \eqref{Eax}
the map $\bar q _\ax : E_\ax \ra \monu$ is the covering such that
$\bar{q}_{\ax*}\pi_1(E_\ax, e_\ax)= K_a: = p_*\meno (H_a) $.  Hence
$f_a: K_a\ra G$ gives a morphism $\pi_1(E_\ax,e_\ax) \ra G$ and thus a
pointed $G$-covering
$ \uu_\ax: ( \ccs_\ax, z_\ax ) \ra ( E_\ax, e_\ax) $ such that
$\im u_{a*} = (\bar{q}_{a*})\meno (\ker f_a)$.  In other words,
$u_\ax$ is the covering such that
\begin{gather}
\label{defu}
\im u_{\ax*} = \bar{q}_{a*}\meno(\ker f_a).
\end{gather}
Composing with $\burp_\ax$ we finally get a holomorphic family of
noncompact Riemann surfaces
$\pi_\ax= \burp_\ax\circ \uu_\ax : \ccs_\ax \ra \basa$.  The following
diagram describes the whole situation.
\begin{equation*}
\begin{tikzcd}
( \ccs_\ax, z_a) \arrow{dr}[swap]{\pi_\ax} \arrow{r}{\uu_{\ax}} &(
E_\ax,e_a) \arrow[r, "\bar q_\ax"] \arrow[d, "\burp_\ax"]
& (\M_{0,n+1}, x) \arrow[d, "p" ] \\
& (\basa,y_a) \arrow[r, "q_a"] & (\M_{0,n},X).
\end{tikzcd}
\end{equation*}
It might help to compare this diagram with the corresponding diagram
of groups:
\begin{equation*}
\begin{tikzcd}
\ker f_a \arrow[hookrightarrow]{r} & K_a \arrow{d}{p_*}
\arrow[hookrightarrow]{r}  & K_x  \arrow{d}{p_*} \\
& H_ a \arrow[hookrightarrow]{r} & H_X
\end{tikzcd}
\end{equation*}

Summing up $p$ is the universal family of lines with $n$ holes,
{$q_a$} is a covering used as a base change, $\burp_\ax$ is the
pull-back family of lines with $n$ holes, $\uu_\ax$ is a Galois cover
and $\pi_\ax$ is a family of non-compact Riemann surfaces. Each fibre
of $\pi_\ax$ covers the corresponding fibre of $\burp_\ax$. More
precisely, if $\yy\in \basa$ and $X=q_\ax(\yy) \in {\M_{0,n}}$,
looking at the fibres over $y$ we have the unramified $G$-covering
\begin{gather}
\label{pre-compatt}
\ccs_{\ax,\yy} \lra E_{\ax,\yy} = \PP^1 -X.
\end{gather}

The last step in the construction is the fibrewise compactification,
which is an application of Grauert-Remmert Extension Theorem. (See
\cite[Ch. XII, Thm. 5.4, p. 340]{SGA})

\begin{teo}[Grauert-Remmert Extension Theorem]
\label{grr}Let $Y$ be a connected complex manifold and $Z\subset Y$
a closed analytic subset such that $ Y^\circ:=Y\setminus Z$ is dense
in $Y$. Let $f^\circ: X^\circ \ra Y^\circ$ be a finite unramified
cover. Then up to isomorphism there exists a unique normal analytic
space $X$ and a unique analytic covering $f:X\ra Y$ such that
$X^\circ \subset X$ and $f^\circ=f|_{X^\circ}$.
\end{teo}
\begin{cor}
In the hypotheses above, if $Z$ is a smooth divisor, then $X$ is
smooth.
\end{cor}
\begin{proof}
Let $D$ be the unit disc.  Using a local chart $U\cong D^n$ of $Y$
such that $U\cap Z=D^n\cap\{z_1=0\}$ we get a finite cover of
$D^*\times D^{n-1}$. By the topological classification of coverings
disc, it is of the form $(z_1,...,z_n) \mapsto (z_1^m,z_2,...,z_n)$
for some $m\geq1$, hence extends to an analytic cover $D^n \ra
D^n$. So by uniqueness $f^{-1}(U)\cong D^n$. In particular
$f^{-1}(U)$ is smooth.
\end{proof}

\begin{lemma}
The unramified covering $\uu_{\ax}:\ccs _\ax \ra E_\ax$ extends
uniquely to an algebraic ramified cover
$\uu_{\ax}:\cc _\ax \ra\mathbb{P}^1\times \basa$, with $\cc_{\ax}$
and $\basef_{a}$ smooth and quasi-projective.
\end{lemma}
\begin{proof}
Consider $\mathbb{P}^1 \times \M_{0,n}$. Let $x_0\in \mathbb{P}^1$
and $X=(x_1,...,x_{n})\in \M_{0,n}$. Recall that this means that
$x_{n-2}=0, x_{n-1}=1, x_n = \infty$ and
$(x_1, \lds, x_{n-3})\in \conf_{0,n-3}\cxx$.  Let
$Z_i\subset \mathbb{P}^1 \times \M_{0,n} $ be the smooth divisor
$Z_i:=\{x_0=x_i\}$ for $i=1,...,n$.  The divisors $Z_1, \lds, Z_n$
are pairwise disjoint, so their union, which we denote by $Z$, is a
smooth divisor of $\mathbb{P}^1 \times \M_{0,n}$.  The map
$\bar q _\ax$ in \eqref{Eax} obviously extends to a map
\begin{gather*}
\bar q _\ax: \PP^1 \times \basa \lra \PP^1 \times \mon.
\end{gather*}
Then $\bar q_\ax^*Z$ is a smooth divisor of
$\mathbb{P}^1\times \basef _\ax$.  Since
$\M_{0,n+1}=(\mathbb{P}^1 \times \M_{0,n})\setminus Z$,
$E_\ax=(\mathbb{P}^1\times \basa ) \setminus \bar{q}_a^* Z$. It
follows that we can apply Grauert-Remmert Extension Theorem to the
topological covering $\uu_{\ax}:\ccs_\ax \ra E_\ax$, which can be
thus completed to a ramified cover
$\uu_{\ax}:\cc_\ax \ra\mathbb{P}^1\times \basa$, with $\cc_{\ax}$
smooth.  To prove the quasi-projectivity one uses a similar
argument.  An \'etale analytic cover of a quasiprojective variety is
quasiprojective and the covering map is algebraic, see
e.g. \cite[Ch. XII, Thm. 5.1, p. 333] {SGA}.  Since $\mon$ and
$\monu$ are quasiprojective, and $q_a$ and $\bar{q}_a$ are \'etale,
we get that $\basef_a$ and $E_a$ are quasiprojective and $q_a$ and
$\bar{q}_a$ algebraic morphisms. Let $\bar\basef_a$ be a projective
manifold containing $\basef_a$ as an open subset. Then $E_a$ is a
Zariski open subset of $\puno\times \bar\basef_a$ and we can apply
Grauert-Remmert Extension Theorem to $\uu_{\ax}:\ccs_\ax \ra E_\ax$
this time viewing $E_a$ as an open subset of
$\puno \times \bar{\basef}_a$.  We obtain a ramified cover
$\bar{\uu}_{\ax}:\bar\cc_{\ax}\ra \puno \times \bar\basef_a$. Since
$\puno\times \bar\basef_a$ is projective, $\bar\cc_{\ax}$ is also
projective. By uniqueness
$\cc_{\ax}=\bar{\uu}_{\ax}^{-1}(\puno \times \basef_a)$. So it is
quasi-projective.

\end{proof}

\begin{say}
Notice that the projection
$\mathbb{P}^1\times \basa \rightarrow \basa$ extends extends
$\burp_\ax$ in \eqref{Eax}, while the composition
$ \cc_{\ax}\overset{\uu_{\ax}}{\longrightarrow} \mathbb{P}^1\times
\basa 
{\longrightarrow} \basa$, extends $\pi_\ax$.  We denote the
extensions by the same symbol. We claim that
\begin{gather*}
\pi_{\ax}: \cc_{\ax} \longrightarrow \basa
\end{gather*}
is a submersion.  Indeed, let $U\cong D^n$ be a local chart in
$\puno \times \basa$ such that
$U\cap \pi^*Z=U\cap \pi^*Z_i= \{x_0-x_i=0\}$ for some $i=1,...,n$
(with $x_{n-2}=0$, $x_{n-1}=1$, $x_n=\infty$). Denoted by
$w=x_0-x_i$, we get that $w, x_1,...,x_{n}$ are local coordinates on
$U$ and 
$\pi'|_{\pi'^{-1}(U)}:\pi'^{-1}(U)\ra U$ is of the form
$(w, x_1,...,x_{n})\mapsto (w^m, x_1,...,x_{n})$, for some
$m\geq 2$. We conclude that locally
$\pi_{\ax} (w, x_1,...,x_{n}) = (x_1,...,x_{n})$.  Thus $\pi_{\ax}$
is a submersion onto a smooth base and its fibers are smooth curves.
\end{say}
\ \\
If $\yy\in \basa$, the fibre $ \cc_{\ax,\yy} \ra \PP^1 $ of $\pi_\ax$
over $y$ is the unique smooth compactification of the unramified cover
\eqref{pre-compatt}, i.e. the one given by Riemann's Existence
Theorem.

We call
\begin{equation*}
\begin{tikzcd}
\cc_\ax \arrow{rr}{\uu_\ax} \arrow{dr}{\pi_\ax} &&
\PP^1 \times \basa \arrow{ld}{\burp_\ax} \\
& \basa &
\end{tikzcd}
\end{equation*}
the family of $G$-coverings associated with the datum
$\Datum\in \datan(G)$, the point $x=(x_0,X)\in \monu$, the geometric
basis $\base$ of $\pi_1(\puno-X,x_0)$ and the extension
$a\in \coll(x,\base, \theta)$.

\begin{teo}
\label{costruzione}
If $x \in \monu$, $\base $ is a basis an basis of $N_x$ and
$\theta $ is an $n$-datum, then
\begin{gather}
\label{eq:4}
\kolf (x, \base, \theta) : =\{\cc_a \ra \basef_a \}_{a\in \coll(x,
\base, \theta)}
\end{gather}
is a collection of families in the sense of Definition
\ref{colldef}.
\end{teo}
\begin{proof}
It remains only to prove property (4).  We start with an
observation.  If $p_i : (E_i, e_i) \ra (B, b) $ are coverings and
$\im p_{1*} \subset \im p_{2*}$, the unique continuous map
$f: (E_1,e_1) \ra (E_2, e_2) $ such that $p_2 \circ f = p_1$ is a
covering map.  Indeed let $f: (X,x)\ra (E_2,e_2)$ be the covering
with $\im f_* = p_{2*}\meno (\im p_{1*})$. Then $p_2\circ f$ is a
covering isomorphic to $p_1$, so we can assume $p_1= p_2 \circ f$.

Now, given $a=(H_a, f_a)$ and $b=(H_b , f_b)$, $a\geq b$ means that
$H_a \subset H_b \subset \pi_1(\mon, X)$, hence
$K_a=p_*\meno (H_a) \subset K_b = p_*\meno (H_b)$ and
$f_a: K_a \ra G$ is the restriction of $f_b$.  We have coverings
$q_i:(\basef_i, y_i) \ra (\mon, X)$ with $H_i = \im q_{i*}$ for
$i=a,b$. By the observation at the beginning there is a unique
covering map $v : (\basef_a, y_a) \ra (\basef _b, y_b) $ such that
$q_b \circ v = q_a$ and $\im v_* = q_{b*}\meno (H_a)$.
For the same reason, since $\im \bar{q}_{i*} = K_i$ for $i=a,b$
there is a covering $\tv: (E_a,e_a) \ra (E_b,e_b)$ such that
$\bar{q}_b \circ \tv = \bar{q}_a$.  We claim that
\begin{gather}
\label{commutano}
\burp_b \tv = v \burp _a.
\end{gather}
Indeed
$q_b \burp_b \tv = p\bar{q}_b \tv = p {\bar{q}_a} = q_a \burp_a =
q_b v \burp_a$.  Hence $\burp_b \tv $ and $ v \burp_a$ lift the same
map with respect to the covering $q_b$. Since
$\burp_b \tv (e_a) = y_b = v \burp_a(e_a)$ we conclude that
$\burp_b \tv = v \burp _a$ as claimed.

Finally we have the coverings $u_i :\ccs_i \ra E_i$ such that
$\im u_{i*} = \bar{q}_{i*}\meno(\ker f_i)$, see \eqref{defu}.  Since
$\tv_* = \bar{q}_{b*}\meno \circ \bar{q}_{a*}$ and
$\ker f_a \subset \ker f_b$ we have
$ \tv_* (\bar{q}_{a*}\meno (\ker f_a )) = \bar{q}_{b*}\meno (\ker
f_a ) \subset \bar{q}_{b*}\meno (\ker f_b ) $.  This means that
\begin{gather}\label{umpf-b}
\im (\tv \circ u_a )_* =\tv_* (\bar{q}_{a*}\meno (\ker f_a ))
\subset \im u_{b*} .
\end{gather}
So we can apply once more the observation at the beginning and we
get a covering $\ttv : \ccs_a \ra \ccs_b$ such that
\begin{gather}
\label{umpf-a}
u_b \ttv = \tv u_a \qquad \im \ttv_* = u_{b*}\meno (\im (\tv \circ
u_a)_*).
\end{gather}
Composing with $\burp_a$ and $\burp_b$ and using \eqref{commutano}
we get a commutative diagram
\begin{equation}\label{ccc}
\begin{tikzcd} [row sep =scriptsize]
\ccs_a \arrow[r, "\ttv"] \arrow[d, swap, "\pi_a"]
& \ccs_b\arrow[d, "\pi_b" ] \\
\basef_a \arrow[r, "v"] & \basef_b.
\end{tikzcd}
\end{equation}
with $\pi_a, \pi_b$ bundles and $\tv, \ttv$ coverings. We claim that
\begin{gather}
\label{eq:3}
\im \ttv_* = \pi_{b*}\meno ( \im v_*).
\end{gather}
Indeed starting from \eqref{umpf-b} we compute
\begin{gather*}
\im (\tv \circ u_a)_* = \bar{q}_{b*}\meno (\ker f_a) =
\bar{q}_{b*}\meno ( K_a\cap \ker f_b ) = \bar{q}_{b*}\meno (K_a)
\cap \bar{q}_{b*}\meno (\ker f_b)
\\
\bar{q}_{b*}\meno (\ker f_b) = \im u_{b*}\qquad
K_a = p_*\meno (H_a) \\
\bar{q}_{b*}\meno (K_a) = \bar{q}_{b*} \meno p_*\meno (H_a) =
(p\bar{q}_{b*})_*\meno (H_a) = ( {q}_{b*}\burp_{b*})\meno (H_a) = \\
= \burp_{b*} \meno (q_{b*}\meno (H_a) ) =
\burp_{b*} \meno (\im v_* ) \\
\im (\tv \circ u_a)_*
= \burp_{b*} \meno (\im v_* ) \cap \im u_{b*}.
\end{gather*}
So from \eqref{umpf-a} we get
\begin{gather*}
\im \ttv_* = u_{b*}\meno (\im (\tv \circ u_a)_*) = u_{b*}\meno (
\burp_{b*} \meno (\im v_* ) \cap \im u_{b*} ) =
u_{b*}\meno ( \burp_{b*} \meno (\im v_* )  ) = \\
=\pi_{b*}\meno (\im v_*).
\end{gather*}
This proves \eqref{eq:3}.  Applying Lemma \ref {converse} to the
diagram \eqref{ccc} we get that $\ttv: \ccs_a \ra v^*\ccs_b$ is an
isomorphism of bundles over $\basef_a$.  The map $\ttv$ is an
isomorphism of the coverings $\ccs_a \ra E_a$ and
$v^*\ccs_b \ra E_a$. By the uniqueness statement in Grauert-Remmert
Extension Theorem, it extends to an isomorphism of the coverings
$\cc_a \ra \PP^1\times \basef_a$ and
$v^*\cc_b \ra \PP^1\times \basef_a$.  This extension is an
isomorphism of the families of curves: $\cc_a \cong v^*\cc_b$.
\end{proof}

\section{Independence from the choices}
\label{sec:indip-choices}

In this section, we conclude the proof of our Theorem \ref{mainz}. We
present two main arguments. The first one is Theorem \ref{maindep},
whose proof will take up most of the section. It states the
independence of the collection $\kolf (x, \base, \theta)$ on
$x \in \monu$, on the geometric basis $\base$, and on the
$\Aut G \times \auta$- orbit of $ \theta$. Secondly, we show (Theorem
\ref{famtipotopo}) that every curve in a family of the collection
$\pi_a : \cc_a \ra \basef_a$ has the topological type associated with
$\theta$, and that, conversely, for any $G$-curve $C$ with topological
type $[\theta]$, there is at least one fiber of $\cc_a \ra \basef_a$
which is (unmarkedly) $G$-isomorphic to $C$ (and there are only
finitely many such fibers).

\begin{defin}\label{kolfequiv}
We say that two collections of families
$\{ \cc_a \ra \basef_a\}_{a\in \coll}$ and
$\{ \ove{\cc}_\bara \ra \ove{\basef}_\bara\}_{\bara\in \ove \coll}$
are \emph{equivalent} if there is an order preserving bijection
$ a\mapsto \bara $ of $\coll$ onto $\ove\coll$ and for every
$a\in \coll$ a biholomorphism $\w_a: \basef_a \ra \ove\basef_\bara$
such that
\begin{enumerate}
\item $\cc_a\cong \w_a^* \ove\cc_\bara$;
\item if $a, b\in \coll$ and $a\geq b$, the following diagram
commutes:
\begin{equation*}
\begin{tikzcd}
\basef_a \arrow{d}[swap]{v_{ab}} \arrow{r}{\w_a} & \ove\basef_\bara \arrow{d}{\bar v_{\bara \barb}}\\
\basef_b \arrow{r}{\w_b} & \ove\basef_{\bar b} .
\end{tikzcd}
\end{equation*}
\end{enumerate}

\end{defin}

In the following we conclude the independence of our collection from
the choices made: different choices yield to equivalent collections.
\begin{teo}
\label{maindep}
Up to equivalence the collection of families
$\kolf(x , \base, \theta)$ is independent of the choices of $x$ and
$\base$ and only depends on the $\Aut G \times \auta$- orbit of
$ \theta$.  In particular the collection $\kolf (x, \base, \theta)$
only depends on the topological type $[\theta]$.
\end{teo}

The proof of Theorem \ref{maindep} is organized as follows: we start
by showing that the action of $\Aut G$ on $\theta$ does not change the
collection (Lemma \ref{lemA}); next we prove that changing $x$ and
$\base$ by parallel transport leads to equivalent collections (Lemma
\ref{fbetacollf}). The combination of these two results implies that,
up to equivalence, the collection of families
$\kolf(x , \base, \theta)$ does not change under the action
$\Aut G \times \Aut^{**}\Gamma_n$ on $ \theta$ (Lemma
\ref{autGxaut**}).  Finally we combine these results together and
complete the proof of Theorem \ref{maindep}.

\begin{lemma}\label{lemA}
Let $\theta \in \data^n(G)$ and $\eta\in \Aut G$. Set
$\bar \theta : = \eta \circ \theta $. Let
$\coll(x,\base,\theta) \ra \coll(x,\base, \bar\theta) , a\mapsto
\bar a $ be the bijection of Lemma \ref{azioneautG}. Then
$\basef_{\bar a } = \basef _a $ and $\cc_{\bar a} = \cc_a$. So
$\kolf(x,\base,\theta)=\kolf(x,\base, \bar\theta) $. In particular,
for $\z\in N_x$,
$\kolf(x,\base,\theta)=\kolf(x,\base, \theta\circ \inn_\z). $
\end{lemma}
\begin{proof}
Let $\chi: \Ga_n \ra N_x$ be the isomorphism induced from the basis
$\base$. Set $f:= \theta \circ \chi\meno$ and
$\bar f := \bar \theta \circ \chi\meno =\eta\circ f$. By Lemma \ref
{azioneautG} we get a bijective correspondence
$I(x,\base,\theta) \ra I(x,\base, \bar\theta) $ which sends
$a=(H_a,f_a) $ to $ \bar a :=( H_{\bar a}, f_{\bar a})$, where
$ H_{\bar a} = H_a$, and $f_{\bar a} = \eta \circ f_a $.  It follows
that $K_{\bar a} = K_a$ and $\ker f_{\bar a} = \ker f_a$. Therefore
$\basef_{\bar a } = \basef _a, E_{\bar a } = E_a$,
$\ccs_{\bar a} = \ccs_a$ and $\cc_{\bar a} = \cc_a$. For the last
statement, just observe that
$\theta\circ \inn_z=\inn_{\theta(z)}\circ \theta$.
\end{proof}

\begin{lemma}
\label{fbetacollf}
Let $\theta \in \data^n(G)$ and $x, x'\in \monu$.  Let the notation
be as in Proposition \ref{esp-parall}: $\beta $ is a path in $\mon$
from $X$ to $X'$, $f^\beta$ represents the parallel transport along
$\beta$, $f^\beta (x_0)= x_0'$,
$\tilde{\beta}(t) = \tilde{H}(t,x_0)$. Then the collections
$\kolf(x,\base,\theta)$ and $\kolf(x',f^\beta_*(\base),\theta)$ are
equivalent.
\end{lemma}
\begin{proof}
Let $\chi: \Ga_n \ra N_x$ be the isomorphism induced from the basis
$\base$. Set $f:= \theta \circ \chi\meno$ and
$\bar {f}:= f \circ (f^\beta_*)\meno$.  We show that, if
$a\in \coll(x, \base, \theta)$ and $\bar{a} = \Phi(a)$, where $\Phi$
is the map in \eqref{rule}, then the families $\cc_a \ra \basef_a$
and $\cc_\bara \ra \basef_\bara$ are canonically isomorphic.
Consider the diagram
\begin{equation*}
\begin{tikzcd}[column sep=small, row sep=scriptsize]
(*_x) & 1 \arrow{r}& N_x \arrow{d}{f^\beta_*} \arrow{rr}{i_*}
& & K_x \arrow{rr}{p_*} \arrow{d}{\tilde{\beta}_\#}{} && H_X \arrow{d}{\beta_\#} \arrow{r} & 1\\
(*_{x'}) & 1\arrow{r} & {N_{x'}} \arrow{rr}{{i_*}} & &{K_{x'}}
\arrow{rr}{{p}_*} &&H_{X'} \arrow{r}& 1
\end{tikzcd}
\end{equation*}
Assume $a=(H_a, f_a)$ and $\bar{a} = (H_\bara, f_\bara)$.  By the
definition of $\Phi$ we have $H_\bara = \beta_\# (H_a)$,
$K_\bara = \tilde{\beta}_\# (K_a)$,
$f_\bara = f_a \circ ( \tilde{\beta}_\#)\meno$,
$\ker f_\bara = \tilde{\beta}_\# (\ker f_a)$.  It follows from \ref
{ex-pippone} that there are canonical isomorphisms
$\basef_{\bar{a}} \cong \basef_a$, $ E_{\bar{a}} \cong E_a$,
$ \ccs_{\bar{a}} \cong \ccs_a$. By compactifying we get that the
families $\cc_a \ra \basef_a$ and $\cc_\bara \ra \basef_\bara$ are
isomorphic.
\end{proof}

\begin{lemma}\label{autGxaut**}
Let $(\eta,\nu) \in \Aut G \times \Aut^{**} \Ga_n$.  Then the
collections $\kolf (x, \base, \theta)$ and
$\kolf (x, \base,\eta\circ \theta \circ \nu\meno)$ are equivalent.
\end{lemma}

\begin{proof}
We have
$ \bar\nu:= \chi \circ \nu \circ \chi\meno \in \Aut^{**} N_x$.  Set
$\bar{\theta} : = \eta\circ \theta \circ \nu\meno$,
$f:=\theta\circ \chi\meno : N_x \tra G$, and
$\bar {f} := \bar \theta \circ \chi \meno = \eta \circ f\circ \bar
\nu\meno$.  By Proposition \ref{traspa}, there are
$[\alpha]\in \pi_1(\mon, X)$, a lifting $\tilde \alpha $ of $\alpha$
with $\tilde\alpha(0) = \tilde\alpha(1)=x_0$, and a parallel
transport $f_t^\alpha $ such that
$f_t^\alpha (x_0) = \tilde\alpha(t)$ and
$\z\in \pi_1(\PP^1-X , x_0)$ such that
$\bar{\nu} = \inn_\z \circ f^{\alpha}_*$. Note that, in particular,
$f^\alpha (x_0) = x_0$. We get
$\bar f=\eta \circ f \circ (f^{\alpha}_*)^{-1}\circ
\inn_{z^{-1}}$. The statement follows from the previous two Lemmas.
\end{proof}

\begin{proof}[Proof of Theorem \ref{maindep}]

Since changing geometric bases of $N_x$ adapted to $X$ corresponds
to acting with $\Aut^{**}\Gamma_n$, by the previous Lemma it follows
that if the point $x$ is fixed, changing the adapted basis does not
matter. Next fix $x, \bar{x} \in \monu$. Choose a path
$\tilde{\beta} $ in $\monu $ joining $x$ to $\bar{x}$.  Set
$\beta := p\circ \tilde\beta$ and let $f^\beta$ be a parallel
transport such that $f^\beta(x_0) = \bar{x}_0$.  Let $\base$ be an
adapted basis at $x$. Then $f^\beta_*\base$ is an adapted basis at
$\bar{x}$.  By Lemma \ref{fbetacollf} we get that $\kolf(x, \base, \theta) $ and $\kolf (\bar{x}, f^\beta_*\base, \theta)$ are equivalent.
In other words we have independence from $x$ and $\base$ as long as $\base $ is adapted to $x$.
We also have that $\theta$ only matters through its $\Aut G \times \Aut^{**}\Ga_n$-orbit by Lemma \ref{autGxaut**}.
It remains to show independence from the $\Aut^*\Ga_n$-orbit.  It
follows from the definitions in \ref{bbbendef} that this is
equivalent to showing that if $x\in \monu$, $\base$ is a basis
adapted to $x$ and $\bbase$ is an arbitrary basis of
$\pi_1(\puno-X,x_0)$, then the collections $\kolf(x, \base, \theta)$
and $\kolf (x, \bbase, \theta)$ are equivalent.  Let us prove this
statement.  There is a permutation $\sigma \in S_n$ such that
$\bbase$ is adapted to $(x_0, x_{\sigma_1}, \lds, x_{\sigma_n})$.
Define
\begin{gather*}
\tau: \mon \lra \mon, \tau(x_1, \lds, x_n) := (x_{\sigma_1}, \lds, x_{\sigma_n})\\
\tilde{\tau}: \monu \lra \monu, \tau(x_0,x_1, \lds, x_n) : =
(x_0,x_{\sigma_1}, \lds, x_{\sigma_n}).
\end{gather*}
Set $\bar{x} = \tilde{\tau}(x)$ and $\bar{X} = \tau(X)$.  By the
previous results we know that $\kolf(x,\base,\theta)$ and
$\kolf(\bar{x}, \bbase, \theta)$ are equivalent.  It remains to
check that also $\kolf(\bar{x}, \bbase, \theta)$ and
$\kolf(x, \bbase, \theta)$ are equivalent.  Consider the diagram:
\begin{equation*}
\begin{tikzcd}[column sep=small, row sep=scriptsize]
(*_x) & 1 \arrow{r}& N_x \arrow{d}{\id_{N_x}} \arrow{rr}{i_*}
& & K_x \arrow{rr}{p_*} \arrow{d}{\tilde{\tau }_*} && H_X \arrow{d}{\tau_*} \arrow{r} & 1\\
(*_{\bar{x}}) & 1\arrow{r} & {N_{\bar{x}}} \arrow{rr}{{i_*}} &
&{K_{\bar{x}}} \arrow{rr}{{p}_*} & & H_{\bar{X}} \arrow{r}& 1
\end{tikzcd}
\end{equation*}
To check commutativity: observe that $\tilde{\tau}$ sends the fiber
over $X$ to the fiber over
$\bar{X}:= (x_{\sigma_1}, \lds, x_{\sigma_n})$, i.e.
$\tilde{\tau} (\PP^1-X) \times \{X\} = (\PP^1-X) \times \{\bar{X}\}
$ and on the first factor it is the identity map.  We use this
diagram with $f= \bar{f} =\theta \circ \bar{\chi}\meno$. We get the
usual correspondence $a\mapsto \bara$,
$\coll(x,\bbase,\theta) \ra \coll(\barx, \bbase, \theta)$ with
\begin{gather}
\label{auffa}
H_\bara = \tau_*(H_a), \qquad K_\bara = \tilde{\tau}_*(K_a),
\qquad \ker f_\bara = \tilde{\tau}_*(\ker f_a) .
\end{gather}
Consider the diagram
\begin{equation*}
\begin{tikzcd}[sep = .35 cm]
& (\ccs_\bara ,z_\bara)\arrow{rr}{u_\bara} & &(E_\bara,e_\bara)
\arrow{dd}[near end]{\burp_\bara}
\arrow{rr}{\bar{q}_\bara} &   & \monu \arrow{dd}{p} \\
(\ccs_a,z_a) \arrow{rr}{u_a} \arrow[dashed]{ru}{\hat{\w}_a} & &
(E_a,e_a) \arrow{dd}{\burp_a} \arrow[dashed]{ru}{\tilde{\w}_a}
\arrow[crossing over]{rr}[near end]{\bar{q}_a} & & \monu
\arrow{ru}{\tilde{\tau}}   & \\
&&  & (Y_\bara, y_\bara)   \arrow{rr}[near end]{q_\bara} & & (\mon, \bar X)    \\
&& (Y_a,y_a) \arrow[dashed]{ru}[swap]{\w_a}
\arrow{rr}[swap]{q_a} & & (\mon, X) \arrow[crossing over,
leftarrow]{uu}[near end,swap]{p} \arrow{ru}{\tau} &
\end{tikzcd}
\end{equation*}
By a repeated use of the lifting theorem and using \eqref{auffa} we
can show the existence of homeomorphisms $\w_a, \tilde{\w}_a$ and
$\hat {\w}_a$ making the diagram commute. Indeed
$(\im (\tau \circ q_{a})_*) =\tau_*(H_a)= \im q_{\bara *}$ by the
first equation in \eqref{auffa}. So $\w_a$ is the isomorphims
between the pointed coverings $\tau \circ q_a$ and $q_\bara$. By the
same argument, using the second equation in \eqref{auffa}, we get
the isomorphism $\tilde{\w}_a$.  Consider the cube on the right in
the diagram.  All its faces (except the left one) commute. But then
$ q_\bara \burp_\bara \tilde{\w}_a = p \bar{q}_\bara \tilde{w} _a =
p \tilde{\tau} \bar{q}_a = \tau p \bar{q}_a = \tau q_a \burp_a =
q_\bara \w_a \burp_a$.  So $\burp_\bara \tilde{\w}_a $ and
$\w_a \burp_a$ lift the same map with respect to $q_\bara$. Since
$\burp_\bara \tilde{\w}_a (e_a)= y_\bara = \w_a \burp_a (e_a)$ we
conclude that $\burp_\bara \tilde{\w}_a = \w_a \burp_a$.

Finally consider the horizontal square on the left of the
diagram. We want to show that
$\im (\tilde{\w}_a \circ u_a) _* = \im u_{\bara *}$. We compose with
the injective morphism $\bar{q}_{\bara *}$ and compute
$ \bar{q}_{\bara *} ( \im (\tilde{\w}_a \circ u_a) _*) =
\tilde{\tau}_*q_{a*} ( \im u_{a*}) = \tilde{\tau}_* ( \ker f_a)$. By
the third equation in \eqref{auffa} this equals
$ \ker f_\bara = \bar{q}_{\bara *} ( \im u_{\bara *})$.  Thus
$ \bar{q}_{\bara *} ( \im (\tilde{\w}_a \circ u_a) _*) =
\bar{q}_{\bara *} ( \im u_{\bara *})$ and
$\im (\tilde{\w}_a \circ u_a) _* = \im u_{\bara *}$. So the lifting
theorem again yields existence of an isomorphism $\hat{\w}_a$ making
everything commutative. The homeomorphisms $\w_a, \tilde{\w}_a $ and
$\hat {\w}_a$ are in fact biholomorphisms as observed in \ref
{ex-pippone}. It follows that $\pi_\bara \hat{\w}_a = \w_a \pi_a$.
By the uniqueness statement in Grauert-Remmert Extension Theorem,
$\tilde{\w}_a$ extends to a biholomorphism between $\cc_a$ and
$\cc_\bara$.  Thus $\cc_a \cong \tilde{\w}_a^* \cc_\bara$.

Property (2) in the Definition \ref{kolfequiv} follows again by the
lifting theorem:
\begin{equation*}
\begin{tikzcd}
\basef_a \arrow{rrd}{v_{ab}} \arrow{rr}{\w_a} \arrow[swap]{rdd}{q_a}&  & \basef_\bara \arrow{ddr}[near end,swap]{q_\bara}    \arrow[yshift=0.3ex,xshift=-0.2ex]{rrd}{\bar{v}_{\bara \bar{b}} } && &&  \\
& & \basef_b\arrow[crossing over]{rr}{\w_b} \arrow{ld}{q_b} && \basef_{\bar{b}} \arrow{ld}{q_{\bar{b}}} && \\
& \mon \arrow[swap]{rr}{\tau} & & \mon &&&
\end{tikzcd}
\end{equation*}
we have
$\bar{q}_{\bar{b}} \bar{v}_{\bara \bar{b}} \w_a = q_\bara \w_a =\tau
q_a =\tau q_b v_{ab} = q_{\bar{b}} \w_b v_{ab}$. So
$ \bar{v}_{ab} \w_a$ and $ \w_a v_{ab}$ lift the same map. Moreover
$ \bar{v}_{\bara \bar{b}} \w_a (y_a) = \bar{v}_{ab}( y_\bara)=
y_{\bar{b}}= \w_b (y_b ) =\w_b v_{ab} (y_a)$. So the two maps
coincide: $ \bar{v}_{ab} \w_a = \w_a v_{ab}$. This proves (2).
\end{proof}

\begin{teo}\label{famtipotopo}
Let $G$ be a finite group and $\theta\in \datan(G)$. Choose a point
$x\in \M_{0,n+1}$ and a geometric basis $\base$ of $N_x$. Let
$\pi_a : \cc_a \ra \basef_a$ be any family in the collection
$\kolf (x, \base, \theta)$. Then every curve in the family has the
topological type given by
$[\theta]\in \datan(G)/\Aut G \times \auta$. Conversely, every
algebraic curve with a $G$-action of the topological type given by
$[\theta]$ is (unmarkedly) $G$-isomorphic to some fiber. Moreover,
there are only finitely many such fibers.
\end{teo}

\begin{proof}
Consider $\pi_a : \cc_a \ra \basef_a$ and let $y,y'\in
\basef_a$. Let $\beta $ be a path in $\basef_a$ from $y$ to $y'$,
and let $f^\beta$ represent the parallel transport along $\beta$. By
Lemma \ref{lemmatrasporto}, we get a $G$-equivariant diffeomorphism
$\cc_{\ax,\yy}\ra \cc_{\ax,\yy'}$. Hence the $G$-actions on
$\cc_{\ax,\yy}$ and $\cc_{\ax,\yy'}$ have the same topological
type. This proves the first statement. Now let $C$ be an algebraic
curve such that $G$ acts effectively on $C$ in such a way that
$C/G \cong \puno$. We get the ramified covering $\pi : C \ra
\PP^1$. By acting via $\PGL(2, \C)$, one can move any three branch
points of $\pi$ to $0,1$ and $\infty$. We can thus assume that the
set of critical values of $\pi : C \ra \PP^1$ coincides with
$Y\in \M_{0,n}$.  Set $C^*:=\pi^{-1}(\puno \setminus Y)$.  Fix a
point $y_0\in \puno-Y$ and consider the monodromy
$f: \pi_1(\puno-Y, y_0)\ra G$ associated with
$\pi|_{C^*}:C^*\ra \puno \setminus Y$. Finally fix a basis
${\base}'$ of $\pi_1(\puno-Y, y_0)$ to $Y$. Let
$\chi:\Gamma_n\ra\pi_1(\puno-Y, y_0)$ denote the associated
isomorphism. Denote by $\theta'=f \circ \chi: \Gamma_n\ra G$ the
datum associated with $C$. We get a collection
$\kolf (y, \base', \theta')$. Assume that $C$ has the same
topological type of $G$-action as $[\theta]$, namely that
$[\theta]=[\theta']\in \datan(G)/\Aut^{**}\Gamma_n\times \Aut G$. By
Theorem \ref{maindep} the collections $\kolf (x, \base, \theta)$ and
$\kolf (y, \base', \theta')$ are equivalent. Thus there exist
$\bara\in \coll(y, \base', \theta')$ and a biholomorphism
$\w_a: \basef_a \ra \ove\basef_\bara$ as in Definition
\ref{kolfequiv}. In particular, $\cc_a\cong \w_a^*
\ove\cc_\bara$. It follows that $C$, which is the central fiber for
$\pi_{\bara}:\cc_\bara\ra \basef_\bara$, is $G$-isomorphic to some
fiber of $\pi_a : \cc_a \ra \basef_a$.  To check that only finitely
many fibers can be $G$-isomorphic to $C$ we argue as follows.  For
any $\sigma \in S_n$ there is a unique $g_\sigma \in \Aut \PP^1$
such that $g_\sigma (y_{\sigma_ {n-2}} ) = 0$,
$g_\sigma (y_{\sigma_ {n-1}} ) = 1$,
$g_\sigma (y_{\sigma_ {n}} ) = \infty$. If $f : C \ra \cc_{a,y}$ is
a $G$-isomorphism for some $y\in \basef_a$, then $f$ descends to an
isomorphism $\bar f\in \Aut \PP^1$ that maps branch points to branch
points. So if $ X := q_a (y)$, we have
$\bar f(\{y_1, \lds, y_n\}) = \{x_1, \lds, x_n\}$. Then there is a
permutation $\sigma$ such that $\bar f(y_{\sigma_i}) = x_i$ for any
$i=1, \lds, n$. So $\bar f= g_\sigma$ and
$X= (g_\sigma(y_1), \lds, g_\sigma(y_n))$. This shows that there is
a finite number of possibilities for $X$, so a finite number of
possibilities for $y$ since $q_a$ is finite.
\end{proof}

\section{The centerless case}\label{sec:centerless}

If the group $G$ has trivial center, the whole discussion in Sections
\ref{sec:groups},\ref{sec:famG-curve} and \ref{sec:indip-choices} is
greatly simplified.

Indeed, let us go back to the setting at the beginning of Section
\ref{sec:groups} and let us consider again the sequence $(*)$.

\begin{teo}\label{centerless}
If the sequence $(*)$ at p. \pageref{pagestar} splits and
$Z(G)=\{1\}$, then there exists a minimum $a_{min}\in \collf$ and it
is unique.
\end{teo}
\begin{proof}
With the notation of Lemma \ref{grupplemma}, set
$H''':=\{h\in H'':\ \tilde\varepsilon_h\in \Inn G\}$. Note that
$H'\subset H'''\subset H''$ and that $H'''$ has finite index in
$H''$ and in $H$.  By assumption the map $G \ra \Inn G$ is
bijective.  So for every $h\in H'''$, there is a unique element of
$G$, denoted by $\phi(h)$, such that
$\tilde \varepsilon_h=\inn_{\phi(h)}$. We get a map
$\phi: H'''\ra G$. Since $\tilde \varepsilon$ is a morphism, we have
$\inn_{\phi(hh')}=\inn_{\phi(h)\phi(h')}$ and, since $Z(G)=\{1\}$,
this implies that $\phi$ is a morphism. Also, by construction,
$\phi$ satisfies $\inn_{\phi(h)}\circ f=f\circ
\varepsilon_h$. Therefore, by Lemma \ref{lemmagruppi}, there exists
a morphism $\tilde f: N\rtimes_\eps H''' \ra G$ extending $f$ such
that $\tilde f|_{H'''}=\phi$. Thus $(H''',\tilde f)\in
\collf$. Moreover, since $\phi$ is unique, so is $\tilde f$. Now let
$a=(H_a, f_a)\in \collf$ and observe that, by Lemma
\ref{lemmagruppi}, every $h\in H_a$ satisfies \eqref{exmor}. It
follows that $H_a\subset H'''$ and $\phi_a=\phi|_{H_a}$ and thus we
conclude that $a=(H_a, f_a)\geq (H''', \tilde f)$. Uniqueness of the
minimum in obvious in any ordered set.
\end{proof}

Next let $N_x$, $K_x$ and $H_X$ be as in \eqref{notax} and consider
the splitting exact sequence $(*_x)$. As usual, choose a geometric
basis $\mathscr{B}=\{[\alfa_i]\}_{i=1}^n$ of $N_x$, let
$\chi: \Ga_n \ra N_x$ be the isomorphism induced from the basis
$\base$, and, for a datum $\theta: N_x\ra G$, set
$f:= \theta \circ \chi\meno: N_x\ra G $. Theorem \ref{centerless}
applied to $(*_x)$ reads as follows:

\begin{teo}\label{casocenterless}
If $G$ has trivial center, then there exists a minimum
$\ax_{min}\in\coll(x,\base, \theta)$ and it is unique.
\end{teo}

Thus in this case by choosing the minimum we have a canonical choice
of a family.
%
Thus, if the center of $G$ is trivial, the choice of a point $x\in
\M_{0,n+1}$, a geometric basis
$\mathscr{B}=\{[\alfa_i]\}_{i=1}^n$, and a datum $\theta: N_x\ra
G$ yields a well-defined minimum family \begin{gather*} \pi_{(x,\base,
\theta)} : \cc_{(x,\base, \theta)}\ra \basef_{(x,\base, \theta)},
\end{gather*}
and we can forget about the whole collection.  Moreover by Theorem
\ref{maindep} changing $x$ or $\mathscr{B}$ or
$\theta$ inside its $\Aut G \times
\auta$- orbit amounts to passing from a collection to an equivalent
one. Since equivalence is order preserving it naturally maps the
minimum to the minumum. This yields the following.

\begin{teo} If
$G$ has trivial center, then up to isomorphism the family
$\pi_{(x,\base, \theta)} : \cc_{(x,\base, \theta)}\ra
\basef_{(x,\base, \theta)}$ is independent of the choices of
$x$ and $\base$ and only depends on the $\Aut G \times
\auta$- orbit of $
\theta$.  In particular the family $\pi_{(x,\base, \theta)} :
\cc_{(x,\base, \theta)}\ra \basef_{(x,\base,
\theta)}$ only depends on the topological type $[\theta]$.
\end{teo}

\section{The abelian case}\label{sec:abelian}

We conclude looking at the special case where the group $G$ is
abelian, the opposite of $G$ being centerless.
\begin{teo}\label{casoabeliano}
If $G$ is abelian, then there exists $\ax\in \coll(x,\base, \theta)$
such that $\basef_{\ax}=\M_{0,n}$.
\end{teo}
\begin{proof}
Let $N_x$, $K_x$ and $H_X$ be as in \eqref{notax} and consider the
splitting exact sequence $(*_x)$, i.e. the top row of
\eqref{seqpuni}. Let $\chi: \Ga_n \ra N_x$ be the isomorphism
induced from the basis $\base$. Set
$f:= \theta \circ \chi\meno: N_x\ra G $.  Now let $\phi: H_X\ra G$
be any morphism.  Let \begin{gather*} \varepsilon:
\pi_1(\M_{0,n})\ra \Aut(\pi_1(\puno-X, x_0))
\end{gather*}
denote the morphism giving the semidirect product in $(*_x)$. By the
considerations in \ref{succesioniesatte}, $\varepsilon$ is just the
restriction to $\pi_1(\M_{0,n})$ of the morphism $\tilde\varepsilon$
giving the splitting of the exact sequence in the second row of
\eqref{seqpuni}.  In \cite[Corollary 1.8.3]{birman} it is
explicitely described the image via $\tilde \varepsilon$ of the
generators of the pure braid group of $n-1$ strings of the plane. To
be more precise, the notation in \cite{birman} corresponds to
identify
\begin{gather*}
\M_{0,n}\cong\{(x_1,...,x_{n-1})\in \conf_{0,n-1}\C: \ x_1=0,\
x_2=1\}
\end{gather*}
instead of \eqref{biolom}.   By this
description one sees that, for a generator $h$ of
$\pi_1 (\M_{0,n})$, $\varepsilon_h$ sends a generator $\ga_j$ of
$\pi_1(\puno-X, x_0)$ to a conjugate of it.   In the setting of
Lemma \ref{lemmagruppi} we have $f\circ \eps_h (\ga_j ) = f(\ga_j)$
since $G$ is abelian.  Similarly $\inn_{\phi(h)} $ is the identity
since $G$ is abelian.  It follows immediately that there exists
$f_a: K_x \ra G$ extending both $f$ and $\phi$. Thus
$(H_x, f_a)\in \coll(x,\base, \theta)$ and this concludes.
\end{proof}

\begin{say}
The proof of Theorem \ref{casoabeliano} shows that, when $G$ is
abelian, for every morphism $\phi: H_X\ra G$ we can build
$f_\phi: K_x \ra G$ extending both $f$ and $\phi$. We point out that
this is the opposite of the uniqueness result in Theorem
\ref{casocenterless}. Of course,
$(H_X, f_\phi)\in \coll(x,\base, \theta)$ is a minimal element for
$(\coll(x,\base, \theta),\geq)$ since $H_X$ is as big as possible,
i.e.  if $b\in \coll(x, \base, \theta)$ and
$ (H_X , f_\phi) \geq b$, then $H_X= H_b$, so $b=(H_x, f_\phi)$.
But different choises of $\phi$ yield elements in
$\coll(x,\base, \theta)$ that are not comparable with respect to the
order relation $\geq$.
\end{say}

\begin{say}\label{esempio}
An important point to stress is that, in the general case,
$H_a \subsetneq H_X$ and $\basef_a \neq \mon$ for every
$\ax\in \coll(x,\base, \theta)$. We now show this via an easy
example.  As in the proof of Theorem \ref{casoabeliano}, we use the
description in \cite{birman} of image via $\tilde \varepsilon$ of
the generators of the pure braid group of the plane and we show
that, in general, there may not exist any morphism
$\tilde f: \pi_1(\mathbb{P}^1\setminus X, x_0)\rtimes H_X\ra G$
extending $f$. Thus, in this case $H_a \subsetneq H_X$ for any
$a\in \coll$. Let $\theta: \Gamma_4\ra S_3$ be given by
$\theta(\gamma_1)=(12)$, $\theta(\gamma_2)=(23)$,
$\theta(\gamma_3)=(23)$, $\theta(\gamma_4)=(12)$. With the notation
in \cite{birman}, $\pi_1(\M_{0,4})$ is free on the generators
$A_{12}$ and $ A_{13}$. We have
$\theta(\varepsilon(A_{12})\gamma_1)=(23)$,
$\theta(\varepsilon(A_{12})\gamma_2)=(13)$,
$\theta(\varepsilon(A_{12})\gamma_3)=(23)$,
$\theta(\varepsilon(A_{12})\gamma_4)=(12)$. Now note that, on one
side, $\gamma_1\gamma_2\gamma_3\gamma_4=1$ and thus
$\ga_1\ga_2\ga_3\ga_4\in\ker\theta$, but on the other side
$\theta(\varepsilon(A_{12})
(\gamma_1\gamma_2\gamma_3\gamma_4))=(23)(13)=(123)\neq 1$.  With the
notation of Lemma \ref{grupplemma} it follows that
$A_{12} \notin H''$, so $H''\neq H_X$.  It follows from Lemma
\ref{lemmagruppi} that for any $a\in \coll$ we have
$H_a \subset H''$. Thus in particular, $H_a\subset H''\neq
H_X$. Thus there do not exist any morphism
$\tilde f: \pi_1(\mathbb{P}^1\setminus X, x_0)\rtimes H_X\ra G$
extending $f$. Geometrically, one can interpret this fact as
follows. On $\M_{0,4}\cong\cxx$ there is the universal family of
elliptic curves $\mathscr{E}\ra \M_{0,4}$. We denote by
$E_{\lambda}$ the fiber of $\mathscr{E}\ra \M_{0,4}$ over
$\lambda\in \cxx$.  The family corresponding to $\theta$ shows that
every elliptic curve has an effective action of $S_3$, which is
built as follows: $S_3 = \Zeta/3 \rtimes \Zeta /2$, where $\Zeta /2$
is the multiplication by $-1$ on $E$ and $\Zeta/3$ is a subgroup of
the translations $(E,+)$. So to build such an action one has to
choose a line inside $E_\la[3]$.  If an extension
$\tilde f:\pi_1(\M_{0,5}) \rightarrow S_3$ exists, then there is a
family of lines
$l_{\lambda}\subset E_{\lambda}[3]\cong H_1(E_{\lambda}, \Zeta/3)$
defined over $\M_{0,4}$. Equivalently, fixing a base point
$\lambda_0\in \M_{0,4}$, there is a line
$l_{\lambda_0}\subset E_{\lambda_0}$ which is stable under the
action of the monodromy of the family $\mathscr{E}$. But the image
of this monodromy is $\Gamma_2$, the congruence subgroup of level 2,
which fixes no line in $H_1(E_{\lambda_0}, \Zeta/3)$.

\end{say}

\begin{say} It follows from the previous remarks that, in the general
case, $\basa$ cannot be $\M_{0,n}$ itself, but is necessarily a
finite cover of it. As pointed out in the introduction, this
corrects an {inaccuracy} in \cite{ganzdiez}.  There it is claimed
that $\basef=\mon$ always. As $\mon$ is birational to
projective space, the authors concluded that the image of the
family in $\M_g$ is always a unirational variety. By Theorem
\ref{casoabeliano} their proof works for abelian covers, hence the
moduli image of a family of abelian covers is always
unirational. In the general case this argument fails and in fact the result is false. Indeed,
Michael D. Fried informed us that he recently found examples of
families for which the moduli image is not unirational. In his
work in progress \cite{friednew}, Fried considers the moduli space
of Galois covers of the line with fixed datum and fixed Nielsen
class.  When a component of this moduli space is of general type
(i.e. a multiple of its canonical class gives an embedding), then
the component is not unirational. When the
datum is for covers with 4 branch points, and the equivalences
include reduction by the action of M\"obius transformations, there
is an explicit formula for the genus of the components (see
\cite{baileyfried02}), which in this case are one-dimensional and
covers of the $j$-line.  When that genus exceeds 1, these spaces
have general type. For the group $A_n$, $n \equiv 1 \mod 4$, and
the branching type of the covers having all four conjugacy classes
$(n+1)/2$-cycles, Fried has computed the components and their
genuses. For $ n$ large, the genus is a nonconstant multiple of
$n^2$. When the equivalence comes from the degree $ n$ permutation
representation of $A_n$, the base $\basef_a$ of any family in the
collection $\{ \cc_a \ra \basef_a\}_{a\in \coll}$ associated with
the datum, has a natural map to one of these components. Thus its
moduli image cannot be unirational. 

\end{say}

\end{document}